\documentclass[oneside,10pt]{amsart}
\makeatletter
\@namedef{subjclassname@2020}{%
  \textup{2020} Mathematics Subject Classification}
\makeatother
\usepackage{amsthm,amssymb}
\usepackage{mathtools}
\usepackage[hidelinks]{hyperref}
\usepackage[a4paper,width=160mm,top=27mm,bottom=27mm]{geometry}
\numberwithin{equation}{section}
\usepackage{todonotes}
\usepackage{color}

\newtheorem{theorem}{Theorem}[section]
\newtheorem{lemma}[theorem]{Lemma}

\newtheorem{proposition}[theorem]{Proposition}
\newtheorem{corollary}[theorem]{Corollary}

\theoremstyle{definition}
\newtheorem{remx}[theorem]{Remark}
\newenvironment{remark}
  {\pushQED{\qed}\remx}
  {\popQED\endremx}

\begin{document}
\address{Yongming Luo
\newline \indent Faculty of Computational Mathematics and Cybernetics \indent
\newline \indent Shenzhen MSU-BIT University, China\indent
}
\email{luo.yongming@smbu.edu.cn}

\newcommand{\diver}{\operatorname{div}}
\newcommand{\lin}{\operatorname{Lin}}
\newcommand{\curl}{\operatorname{curl}}
\newcommand{\ran}{\operatorname{Ran}}
\newcommand{\kernel}{\operatorname{Ker}}
\newcommand{\la}{\langle}
\newcommand{\ra}{\rangle}
\newcommand{\N}{\mathbb{N}}
\newcommand{\R}{\mathbb{R}}
\newcommand{\C}{\mathbb{C}}
\newcommand{\T}{\mathbb{T}}

\newcommand{\ld}{\lambda}
\newcommand{\fai}{\varphi}
\newcommand{\0}{0}
\newcommand{\n}{\mathbf{n}}
\newcommand{\uu}{{\boldsymbol{\mathrm{u}}}}
\newcommand{\UU}{{\boldsymbol{\mathrm{U}}}}
\newcommand{\buu}{\bar{{\boldsymbol{\mathrm{u}}}}}
\newcommand{\ten}{\\[4pt]}
\newcommand{\six}{\\[-3pt]}
\newcommand{\nb}{\nonumber}
\newcommand{\hgamma}{H_{\Gamma}^1(\OO)}
\newcommand{\opert}{O_{\varepsilon,h}}
\newcommand{\barx}{\bar{x}}
\newcommand{\barf}{\bar{f}}
\newcommand{\hatf}{\hat{f}}
\newcommand{\xoneeps}{x_1^{\varepsilon}}
\newcommand{\xh}{x_h}
\newcommand{\scaled}{\nabla_{1,h}}
\newcommand{\scaledb}{\widehat{\nabla}_{1,\gamma}}
\newcommand{\vare}{\varepsilon}
\newcommand{\A}{{\bf{A}}}
\newcommand{\RR}{{\bf{R}}}
\newcommand{\B}{{\bf{B}}}
\newcommand{\CC}{{\bf{C}}}
\newcommand{\D}{{\bf{D}}}
\newcommand{\K}{{\bf{K}}}
\newcommand{\oo}{{\bf{o}}}
\newcommand{\id}{{\bf{Id}}}
\newcommand{\E}{\mathcal{E}}
\newcommand{\ii}{\mathcal{I}}
\newcommand{\sym}{\mathrm{sym}}
\newcommand{\lt}{\left}
\newcommand{\rt}{\right}
\newcommand{\ro}{{\bf{r}}}
\newcommand{\so}{{\bf{s}}}
\newcommand{\e}{{\bf{e}}}
\newcommand{\ww}{{\boldsymbol{\mathrm{w}}}}
\newcommand{\zz}{{\boldsymbol{\mathrm{z}}}}
\newcommand{\U}{{\boldsymbol{\mathrm{U}}}}
\newcommand{\G}{{\boldsymbol{\mathrm{G}}}}
\newcommand{\VV}{{\boldsymbol{\mathrm{V}}}}
\newcommand{\II}{{\boldsymbol{\mathrm{I}}}}
\newcommand{\ZZ}{{\boldsymbol{\mathrm{Z}}}}
\newcommand{\hKK}{{{\bf{K}}}}
\newcommand{\f}{{\bf{f}}}
\newcommand{\g}{{\bf{g}}}
\newcommand{\lkk}{{\bf{k}}}
\newcommand{\tkk}{{\tilde{\bf{k}}}}
\newcommand{\W}{{\boldsymbol{\mathrm{W}}}}
\newcommand{\Y}{{\boldsymbol{\mathrm{Y}}}}
\newcommand{\EE}{{\boldsymbol{\mathrm{E}}}}
\newcommand{\F}{{\bf{F}}}
\newcommand{\spacev}{\mathcal{V}}
\newcommand{\spacevg}{\mathcal{V}^{\gamma}(\Omega\times S)}
\newcommand{\spacevb}{\bar{\mathcal{V}}^{\gamma}(\Omega\times S)}
\newcommand{\spaces}{\mathcal{S}}
\newcommand{\spacesg}{\mathcal{S}^{\gamma}(\Omega\times S)}
\newcommand{\spacesb}{\bar{\mathcal{S}}^{\gamma}(\Omega\times S)}
\newcommand{\skews}{H^1_{\barx,\mathrm{skew}}}
\newcommand{\kk}{\mathcal{K}}
\newcommand{\OO}{O}
\newcommand{\bhe}{{\bf{B}}_{\vare,h}}
\newcommand{\pp}{{\mathbb{P}}}
\newcommand{\ff}{{\mathcal{F}}}
\newcommand{\mWk}{{\mathcal{W}}^{k,2}(\Omega)}
\newcommand{\mWa}{{\mathcal{W}}^{1,2}(\Omega)}
\newcommand{\mWb}{{\mathcal{W}}^{2,2}(\Omega)}
\newcommand{\twos}{\xrightharpoonup{2}}
\newcommand{\twoss}{\xrightarrow{2}}
\newcommand{\bw}{\bar{w}}
\newcommand{\bz}{\bar{{\bf{z}}}}
\newcommand{\tw}{{W}}
\newcommand{\tr}{{{\bf{R}}}}
\newcommand{\tz}{{{\bf{Z}}}}
\newcommand{\lo}{{{\bf{o}}}}
\newcommand{\hoo}{H^1_{00}(0,L)}
\newcommand{\ho}{H^1_{0}(0,L)}
\newcommand{\hotwo}{H^1_{0}(0,L;\R^2)}
\newcommand{\hooo}{H^1_{00}(0,L;\R^2)}
\newcommand{\hhooo}{H^1_{00}(0,1;\R^2)}
\newcommand{\dsp}{d_{S}^{\bot}(\barx)}
\newcommand{\LB}{{\bf{\Lambda}}}
\newcommand{\LL}{\mathbb{L}}
\newcommand{\mL}{\mathcal{L}}
\newcommand{\mhL}{\widehat{\mathcal{L}}}
\newcommand{\loc}{\mathrm{loc}}
\newcommand{\tqq}{\mathcal{Q}^{*}}
\newcommand{\tii}{\mathcal{I}^{*}}
\newcommand{\Mts}{\mathbb{M}}
\newcommand{\pot}{\mathrm{pot}}
\newcommand{\tU}{{\widehat{\bf{U}}}}
\newcommand{\tVV}{{\widehat{\bf{V}}}}
\newcommand{\pt}{\partial}
\newcommand{\bg}{\Big}
\newcommand{\hA}{\widehat{{\bf{A}}}}
\newcommand{\hB}{\widehat{{\bf{B}}}}
\newcommand{\hCC}{\widehat{{\bf{C}}}}
\newcommand{\hD}{\widehat{{\bf{D}}}}
\newcommand{\fder}{\partial^{\mathrm{MD}}}
\newcommand{\Var}{\mathrm{Var}}
\newcommand{\pta}{\partial^{0\bot}}
\newcommand{\ptaj}{(\partial^{0\bot})^*}
\newcommand{\ptb}{\partial^{1\bot}}
\newcommand{\ptbj}{(\partial^{1\bot})^*}
\newcommand{\geg}{\Lambda_\vare}
\newcommand{\tpta}{\tilde{\partial}^{0\bot}}
\newcommand{\tptb}{\tilde{\partial}^{1\bot}}
\newcommand{\ua}{u_\alpha}
\newcommand{\pa}{p\alpha}
\newcommand{\qa}{q(1-\alpha)}
\newcommand{\Qa}{Q_\alpha}
\newcommand{\Qb}{Q_\eta}
\newcommand{\ga}{\gamma_\alpha}
\newcommand{\gb}{\gamma_\eta}
\newcommand{\ta}{\theta_\alpha}
\newcommand{\tb}{\theta_\eta}


\newcommand{\mH}{{E}}
\newcommand{\mN}{{S}}
\newcommand{\mD}{{\mathcal{D}}}
\newcommand{\csob}{\mathcal{S}}
\newcommand{\mA}{{\mathcal{A}}}
\newcommand{\mK}{{K}}
\newcommand{\mS}{{S}}
\newcommand{\mI}{{I}}
\newcommand{\tas}{{2_*}}
\newcommand{\tbs}{{2^*}}
\newcommand{\tm}{{\tilde{m}}}
\newcommand{\tdu}{{\phi}}
\newcommand{\tpsi}{{\tilde{\psi}}}
\newcommand{\Z}{{\mathbb{Z}}}
\newcommand{\tsigma}{{\tilde{\sigma}}}
\newcommand{\tg}{{\tilde{g}}}
\newcommand{\tG}{{\tilde{G}}}
\newcommand{\mM}{{M}}
\newcommand{\mC}{\mathcal{C}}
\newcommand{\wlim}{{\text{w-lim}}\,}
\newcommand{\diag}{L_t^\ba L_x^\br}
\newcommand{\vu}{ u}
\newcommand{\vz}{ z}
\newcommand{\vv}{ v}
\newcommand{\ve}{ e}
\newcommand{\vw}{ w}
\newcommand{\vf}{ f}
\newcommand{\vh}{ h}
\newcommand{\vp}{ \vec P}
\newcommand{\ang}{{\not\negmedspace\nabla}}
\newcommand{\dxy}{\Delta_{x,y}}
\newcommand{\lxy}{L_{x,y}}
\newcommand{\gnsand}{\mathrm{C}_{\mathrm{GN},3d}}
\newcommand{\wmM}{\widehat{{M}}}
\newcommand{\wmH}{\widehat{{E}}}
\newcommand{\wmI}{\widehat{{I}}}
\newcommand{\wmK}{\widehat{{K}}}
\newcommand{\wmN}{\widehat{{S}}}
\newcommand{\wm}{\widehat{\gamma}}
\newcommand{\ba}{\mathbf{a}}
\newcommand{\bb}{\mathbf{b}}
\newcommand{\br}{\mathbf{r}}
\newcommand{\bs}{\mathbf{s}}
\newcommand{\bq}{\mathbf{q}}
\newcommand{\SSS}{\mathcal{S}}
\newcommand{\re}{\mathrm{Re}}
\newcommand{\im}{\mathrm{Im}}
\newcommand{\zt}{{\tilde z}}



\title[A Legendre-Fenchel identity for NLS on $\R^d\times\T^m$]
{A Legendre-Fenchel identity for the nonlinear Schr\"odinger equations on $\mathbb{R}^d\times\mathbb{T}^m$: theory and applications}
\author{Yongming Luo}


\maketitle

\begin{abstract}
The present paper is inspired by a previous work \cite{Luo_Waveguide_MassCritical} of the author, where the large data scattering problem for the focusing cubic nonlinear Schr\"odinger equation (NLS) on $\mathbb{R}^2\times\mathbb{T}$ was studied. Nevertheless, the results from \cite{Luo_Waveguide_MassCritical} are by no means sharp, as we could not even prove the existence of ground state solutions on the formulated threshold. By making use of the variational tools introduced by the author \cite{Luo_inter}, we establish in this paper the sharpened scattering results. Yet due to the mass-critical nature of the model, we encounter the major challenge that the standard scaling arguments fail to perturb the energy functionals. We overcome this difficulty by proving a crucial Legendre-Fenchel identity for the variational problems with prescribed mass and frequency. More precisely, we build up a general framework based on the Legendre-Fenchel identity and show that the much harder or even unsolvable variational problem with prescribed mass, can in fact be equivalently solved by considering the much easier variational problem with prescribed frequency. As an application showing how the geometry of the domain affects the existence of the ground state solutions, we also prove that while all mass-critical ground states on $\mathbb{R}^d$ must possess the fixed mass $\widehat M(Q)$, the existence of mass-critical ground states on $\mathbb{R}^d\times\mathbb{T}$ is ensured for a sequence of mass numbers approaching zero.
\end{abstract}

\section{Introduction}
The paper is devoted to the study of the focusing nonlinear Schr\"odinger equation (NLS)
\begin{align}\label{nls}
(i\pt_t +\Delta_{x,y})u =-|u|^{\alpha}u
\end{align}
and its corresponding solitary wave equation
\begin{align}\label{nls2}
-\Delta_{x,y}u +\omega u=|u|^{\alpha}u.
\end{align}
on the waveguide manifolds $\R_x^d\times\T_y^m$ with $\T=\R/2\pi\Z$. Equation \eqref{nls} arises from several physical applications such as the nonlinear optics and Bose-Einstein condensates \cite{waveguide_ref_1,waveguide_ref_2,waveguide_ref_3} and has been extensively studied in recent years. In addition to its physical significance, the close connection of equation \eqref{nls} with partial differential equations and harmonic analysis has also made it attract much attention from the mathematical community. Different results concerning the well-posedness, long time behavior and existence and stability of the corresponding soliton solutions are nowadays well established. We refer e.g. to the papers \cite{TzvetkovVisciglia2016,TTVproduct2014,Ionescu2,HaniPausader,CubicR2T1Scattering,R1T1Scattering,Cheng_JMAA,ZhaoZheng2021,RmT1,YuYueZhao2021,Barron,BarronChristPausader2021,Luo_Waveguide_MassCritical,Luo_inter,Luo_energy_crit,luo2021sharp,luo2023sure,esfahani2023focusing} for some recent results in this direction.

Among all, we are particularly interested in the problem whether a global solution (as long as it exists) of \eqref{nls} will also be a \textit{scattering} solution, in the sense that the solution of \eqref{nls} resembles the linear solution as time evolves to infinity. While this is expected for NLS on $\R^d$ for suitable nonlinearities, it is not expected for NLS on $\T^m$ due to its periodicity. Thus it becomes a subtle question how the dispersion is balanced when the totally different domains are intertwined with each other.

It turns out that scattering may still happen on the mixed domain $\R^d\times\T^m$ as long as the nonlinearity is chosen appropriately. Heuristically, we may consider the two extremal cases, where
\begin{itemize}
\item either \eqref{nls} is completely independent of the $y$-direction, in which case we demand the nonlinearity to be at least mass-critical w.r.t. the dimension number $d$,
\item or the initial datum of \eqref{nls} has a relatively small support in the $y$-direction so that $\R^d\times\T^m$ may be seen as $\R^{d+m}$ (as it is in some sense a ``large'' domain), in which case the nonlinearity should be at most energy-critical w.r.t. the dimension number $d+m$.
\end{itemize}
Inspired by the heuristics, the number $\alpha$ should hence satisfy $4/d\leq \alpha\leq 4/(d+m-2)$, which particularly implies that $m\in\{0,1,2\}$. Notice that $m=2$ corresponds to the (most difficult) doubly mass- and energy-critical case. As we still do not have a full understanding in the energy-critical model, we will restrict ourselves in this paper to the case $m\in\{0,1\}$ and
\begin{align}\label{add}
\alpha\in[\tas,\tbs),\quad\tas=\frac{4}{d},\quad\tbs=
\left\{
\begin{array}{cl}
\infty,&d+m\leq 2\\
\frac{4}{d+m-2},&d+m\geq 3.
\end{array}
\right.
\end{align}
Moreover, we refer the cases $\alpha=\tas$ and $\alpha=\tbs$ to as the \textit{mass-critical} and \textit{energy-critical} cases respectively, while the interval $(\tas,\tbs)$ is called the \textit{intercritical} regime.

Since the seminal works by Hani-Pausader \cite{HaniPausader} and Tzvetkov-Visciglia \cite{TzvetkovVisciglia2016}, the study on the large data scattering problems for defocusing NLS on $\R^d\times\T^m$ is nowadays complete, see \cite{R2T2,RmT1,CubicR2T1Scattering,R1T1Scattering}. Nevertheless, the focusing type problems had remained open for a long time and has been studied fairly recently by a series papers of the author \cite{Luo_Waveguide_MassCritical,Luo_inter,Luo_energy_crit,luo2021sharp}. Different from the defocusing problems, the main difficulty by studying focusing problems is mainly attributed to a poor understanding in the rather complicated variational structure of the focusing NLS on product spaces. Such issue already arose when the author gave his first attempt on formulating the large data scattering results for the focusing cubic NLS on $\R^2\times\T$ \cite{Luo_Waveguide_MassCritical}. Particularly, the results from \cite{Luo_Waveguide_MassCritical} are by no means sharp, as we could not even prove the existence of ground state solutions on the formulated threshold.

By introducing a general framework involving the so-called \textit{semivirial-vanishing geometry}, the author was able to prove the sharp bifurcation of the scattering and blowing-up solutions for the focusing NLS in the mass-supercritical regime \cite{Luo_inter,Luo_energy_crit,luo2021sharp}. One of the main purposes of the present paper is to show that the theory of the semivirial-vanishing geometry will continue to work in the mass-critical setting and help us to formulate sharp scattering and blow-up results, which will ultimately improve the previous ones established in \cite{Luo_Waveguide_MassCritical}.


We underline, however, that the adaptation of the framework of the semivirial-vanishing geometry into the mass-critical setting is highly non-trivial and technical where several new ideas and tools need be introduced. The main obstacle here is that in the mass-critical setting, the Laplacian and the nonlinear potential share the same scaling, thus we are unable to perturb any useful energy functional by the standard scaling operator $t\mapsto t^\frac{d}{2}u(tx,y)$ which leaves the mass of a function $u$ invariant for all given $t>0$.

The idea for solving this problem is mainly inspired by the recent paper \cite{actionvsenergy} of Dovetta, Serra and Tilli, where the authors studied the duality of the energy and action ground states on $\R^d$. Therein, the authors studied the variational problems\footnote{Throughout the paper, we use hatted notation to denote quantities defined on $\R^d$.}
\begin{equation}\label{1.3}
\begin{array}{c}
\tilde{m}_c:=\inf\{\widehat E(u):u\in H^1(\R^d),\,\widehat M(u)=c\},\\
\\
\tilde{\beta}_\omega:=\inf\{\widehat S_\omega(u):u\in H^1(\R^d)\setminus\{0\},\,\widehat N_\omega(u)=0\}
\end{array}
\end{equation}
for $c,\omega\in\R$, where $\widehat M(u)$, $\widehat E(u)$ denote the mass and energy, $\widehat S_\omega (u)$ the action functional and $\widehat N_\omega(u)$ the constraint deduced from the stationary equation \eqref{nls2}, see Section \ref{sec notation} below for their precise definitions. Optimizers of $\tilde m_c$ and $\tilde \beta_\omega$ are in literature referred to as the \textit{energy} as well as \textit{action} ground states, respectively. The main result in \cite{actionvsenergy} particularly reveals a certain duality between $\tilde m_c$ and $\tilde \beta_\omega$, formulated as a Legendre-Fenchel identity, which reads as follows:

\begin{theorem}[A Legendre-Fenchel identity for $\tilde m_c$ and $\tilde \beta_\omega$, \cite{actionvsenergy}]
For $\alpha\in(0,\tbs)$ and $c\geq 0$ it holds
\begin{align}\label{lf identity rd}
\tilde m_c=\inf_{\omega\in\R}(\tilde \beta_\omega-\frac12 c\omega).
\end{align}
\end{theorem}

We note, however, that despite the elegance of the identity \eqref{lf identity rd}, it is rather less meaningful to apply \eqref{lf identity rd} for studying the mass-(super)critical problems, as in this case we always have $\tilde m_c=-\infty$. To avoid such triviality, we know --- since the seminal paper of Jeanjean \cite{Jeanjean1997} --- that the virial constraint $\widehat K(u)=0$ need be taken into account. More generally, one of the main contributions of the present work is to give the correct formulation of \eqref{lf identity rd} in the mass-(super)critical context for focusing NLS posed on $\R^d\times\T^m$, where the case $m=0$ is also included and the result from \cite{actionvsenergy} is generalized in a correct and non-trivial way.

To formulate our main results, the energy functionals $M(u),E(u),S_\omega(u),K(u),N_\omega(u)$ defined in Section \ref{sec notation} will be invoked. We then define the variational problems
\begin{align*}
m_c&:=\inf\{E(u):u\in H^1(\R^d\times\T^m),\,M(u)=c,\,K(u)=0\},\\
\gamma_\omega&:=\inf\{S_\omega(u):u\in H^1(\R^d\times\T^m)\setminus\{0\},\,K(u)=0\}.
\end{align*}
Similarly, we define the hatted quantities as the ones defined on $\R^d$. Inspired by \eqref{lf identity rd} we may also define the following variational problem
\[\beta_\omega:=\inf\{S_\omega(u):u\in H^1(\R^d\times\T^m)\setminus\{0\},\,N_\omega(u)=0\}.\]
As a first result, the proposition below shows that the variational problems $\gamma_\omega$ and $\beta_\omega$ are indeed equivalent.

\begin{proposition}[Identification of $\gamma_\omega$ and $\beta_\omega$]\label{prop iden of gamma and beta}
For any $\omega\in(0,\infty)$ we have $\gamma_\omega=\beta_\omega$.
\end{proposition}
Proposition \ref{prop iden of gamma and beta} thus enables us to manipulate the analysis solely for the variational problem $\gamma_\omega$, which will be the case in many of the upcoming computations.

Though our main purpose is to establish a generalization of the Legendre-Fenchel identity \eqref{lf identity rd} in the mass-(super)critical setting, we shall see that the existence and periodic dependence of the minimizers of $\gamma_\omega$ will play a pivotal role in the upcoming proofs. Corresponding results can be proved by gluing several ideas from \cite{TTVproduct2014,Luo_inter} and are given as follows:

\begin{theorem}[Existence of ground states for $\gamma_\omega$]\label{thm existence ground states 1}
Let $\alpha\in [\tas,\tbs)$. Then
\begin{itemize}
\item[(i)] For any $\omega>0$ the variational problem $\gamma_\omega$ possesses a positive minimizer $u_\omega$ which also solves \eqref{nls2} with the given $\omega$.

\item[(ii)]In the case $\alpha=\tas$, for any $t>0$ and any optimizer $u_\omega$ of $\gamma_\omega$, the function $v^t$ defined by
\begin{align}\label{1.4}
v^t(x,y)=t^{\frac{d}{2}}u_\omega(tx,y)
\end{align}
is also an optimizer of $\gamma_\omega$.
\end{itemize}
\end{theorem}

\begin{theorem}[$y$-dependence of ground states of $\gamma_\omega$]\label{thm threshold frequency}
Let $\alpha\in [\tas,\tbs)$ and $m\neq 0$. Then there exists some $\omega_*\in(0,\infty)$ such that
\begin{itemize}
\item[(i)] For all $\omega \in (0,\omega_*]$ we have $\gamma_{\omega}=(2\pi)^m \widehat \gamma_\omega$. Moreover, for $\omega \in (0,\omega_*)$ any minimizer $u_\omega$ of $\gamma_\omega$ satisfies $\nabla_y u_\omega= 0$.

\item[(ii)] For all $\omega \in (\omega_*,\infty)$ we have $\gamma_{\omega}<(2\pi)^m \widehat \gamma_\omega$. Moreover, for $\omega \in (\omega_*,\infty)$ any minimizer $u_\omega$ of $\gamma_\omega$ satisfies $\nabla_y u_\omega\neq 0$.
\end{itemize}
\end{theorem}

With the help of Theorem \ref{thm existence ground states 1} and \ref{thm threshold frequency}, we are able to establish the following Legendre-Fenchel identity in the context of the mass-(super)critical model.

\begin{theorem}[The Legendre-Fenchel identity]\label{thm legendre}
Let $\alpha\in[\tas,\tbs)$. For $c>0$, define
\begin{align*}
A_c:=\{\omega>0:\text{ there exists $u\in H^1(\R^d\times\T^m)$ with $u\in\mathrm{argmin}\,\gamma_\omega$ and $M(u)=c$}\},
\end{align*}
where $\mathrm{argmin}\,\gamma_\omega$ denotes the set of minimizers of the variational problem $\gamma_\omega$. Denote also by $Q$ the unique radial and positive solution of \[-\Delta_x Q+Q=Q^{1+\frac{4}{d}}\quad\text{on $\R^d$.}\]
Then
\begin{itemize}
\item[(i)] If $\alpha\in (\tas,\tbs)$ and $m=0$, then $A_c\neq\varnothing$ for any $c\in(0,\infty)$.

\item[(ii)]  There exists
\begin{align*}
(c_n)_n\subset\left\{
\begin{array}{cl}
(0,\infty),&\text{if $\alpha\in(\tas,\tbs)$ and $m=1$},\\
\\
(0,2\pi\widehat M(Q)),&\text{if $\alpha=\tas$ and $m=1$}
\end{array}
\right.
\end{align*}
such that $\lim_{n\to\infty}c_n=0$ and $A_{c_n}\neq\varnothing$ for any $n\in\N$.
\end{itemize}
Moreover, for any $A_c\neq\varnothing$ and $\omega\in A_c$ we have
\begin{align}\label{lengendre identity}
m_c=\gamma_\omega-\frac{1}{2}c\omega.
\end{align}
\end{theorem}

We briefly elaborate on the idea for proving Theorem \ref{thm legendre}. In fact, \eqref{lengendre identity} will immediately follow from the definitions of the variational problems as long as we can prove that $A_c$ is not empty. We will then use Theorem \ref{thm threshold frequency} and suitable test function arguments to prove the latter claim.

\begin{remark}
Since \eqref{lengendre identity} is valid for any $\omega\in A_c$, we may rewrite \eqref{lengendre identity} into
$$m_c=\inf_{\omega\in A_c}(\gamma_\omega-\frac{1}{2}c\omega)=\inf_{\omega\in A_c}(\beta_\omega-\frac{1}{2}c\omega).$$
Notice on the other hand that by definition, the quantity on the r.h.s. of \eqref{lf identity rd} will not exceed the second quantity on the r.h.s. of \eqref{lengendre identity}. One may naturally ask whether \eqref{lf identity rd} will indeed hold in place of \eqref{lengendre identity} for $m_c$ with the additional constraint $K(u)=0$. We note that this can not be the case: as a consequence of Lemma \ref{lemma 4.1} we have $\lim_{\omega\to \infty}(\gamma_\omega-\frac12 c\omega)=\lim_{\omega\to \infty}(\beta_\omega-\frac12 c\omega)=-\infty$, while by definition we always have $m_c\geq 0$. This reveals the fact that \eqref{lengendre identity} improves \eqref{lf identity rd} in a strict sense.
\end{remark}
\begin{remark}
In the case $\alpha=\tas$ and $m=0$, we have a more precise description on the quantities given in Theorem \ref{thm legendre} which may be directly deduced from the classical theories for the mass-critical NLS (see e.g. \cite{weinstein,Kwong_uniqueness}): In fact, we have $A_c=B_c=\varnothing$ for $c<\widehat M(Q)$, $A_c=B_c=(0,\infty)$ for $c=\widehat M(Q)$,
\begin{align*}
m_c=
\left\{
\begin{array}{cl}
\infty,&\text{if $c\in(0,\widehat M(Q))$},\\
\\
0,&\text{if $c=\widehat M(Q)$}
\end{array}
\right.
\end{align*}
and $\gamma_\omega=\frac{1}{2}\omega \widehat M(Q)$ for any $\omega>0$ (where we made the usual convention $\inf\varnothing=\infty$).
\end{remark}

\begin{remark}
In the case $m=1$ we conjecture that $A_c\neq \varnothing $ for any $c\in(0,\infty)$ ($\alpha\in(\tas,\tbs)$) and $c\in(0,2\pi\widehat M(Q))$ ($\alpha=\tas$). The main issue here is that the mapping $\omega\mapsto M(u_\omega)$, where $u_\omega$ is an optimizer of $\gamma_\omega$, is not necessarily well-defined, since $M(u_\omega)$ is not necessarily uniquely determined (there is the $y$-gradient energy competing with the mass in the Pohozaev's identity). In the case $m=0$, however, we may use the classical Pohozaev's identity to show that the value of $M(u_\omega)$ is indeed independent of the choice of the minimizer.
\end{remark}

At the first glance, Theorem \ref{thm legendre} merely confirms the duality of the variational problems $m_c$ and $\gamma_\omega$ in term of the Legendre-Fenchel identity \eqref{lengendre identity}, which seemingly admits limited applications. We will show that despite its concise form, \eqref{lengendre identity} plays a key role for proving results that are barely possible to be shown directly. We recall that we encounter the difficulty that the Laplacian and the nonlinear potential share the same scaling when solving directly the variational problem $m_c$ in the mass-critical setting. With the help of Theorem \ref{thm legendre}, we are in fact able to obtain optimizers of $m_c$ by solving an equivalent problem deduced from \eqref{lengendre identity}. More precisely, we have the following existence result:

\begin{theorem}[Normalized ground states in the mass-critical case]\label{thm threshold mass 2}
Let $\alpha=\tas$ and $m=1$. Then for any $c\in (0,2\pi  \widehat{M}(Q))$ with $A_c\neq\varnothing$ the following statements hold:
\begin{itemize}
\item[(i)]The variational problem $m_c$ has an optimizer $u_c$ that solves \eqref{nls2} with some $\omega
=\omega_c\geq \omega_*$, where $\omega_*$ is the number given by Theorem \ref{thm threshold frequency}.

\item[(ii)]For any $t>0$ and any optimizer $u_c$ of $m_c$, the function $v^t$ defined by
\[v^t(x,y)=t^{\frac{d}{2}}u_c(tx,y)\]
is also an optimizer of $m_c$.

\item[(iii)] Any optimizer $u_c$ of $m_c$ satisfies $\nabla_y u_c\neq 0$.
\item[(iv)] There exist optimizers $u_c^\gamma$ and $u_c^\beta$ of $m_c$ and $\omega_c^\gamma,\omega_c^\beta\geq \omega_*$ such that $u_c^\gamma$ and $u_c^\beta$ are optimizers of $\gamma_{\omega_c^\gamma}$ and $\beta_{\omega_c^\beta}$ respectively.
\end{itemize}
\end{theorem}

\begin{remark}
Using Theorem \ref{thm legendre} we may also establish a similar existence result for the focusing intercritical NLS. Nevertheless, by solving the problem $m_c$ directly a corresponding existence result has already been given in \cite{Luo_inter}. We thus focus in this paper on the mass-critical case and refer to \cite{Luo_inter} for more details about the intercritical model.
\end{remark}

\begin{remark}
Theorem \ref{thm threshold mass 2} implies particularly that the focusing quintic NLS on $\R\times\T$ and focusing cubic NLS on $\R^2\times\T$, which being physically relevant cases, possess normalized ground state solutions which have mass arbitrarily close to zero and have non-trivial $y$-portion.
\end{remark}

Notice again that due to the failure of the scaling arguments, it is rather difficult to find for a given number $c\in (0,2\pi\widehat M(Q))$ a candidate $u\in H^1(\R^d\times\T)$ satisfying $M(u)=c$ and $K(u)=0$. It turns out that Lemma \ref{lemma 4.1} given below, which being a crucial ingredient for proving Theorem \ref{thm legendre}, will indeed guarantee the existence of such a candidate. In conjunction with suitable surgery arguments, we are then able to prove the following dynamical properties of the mapping $c\mapsto m_c$ which play a central role in proving the large data scattering results.

\begin{proposition}\label{prop monotone}
Let $\alpha=\tas$ and $m=1$. Then the mapping $c\mapsto m_c$ is monotone decreasing and lower semicontinuous on $(0,2\pi\widehat M(Q))$.
\end{proposition}

Finally, following the same arguments as the ones given in \cite{Luo_JFA_2022,Luo_DoubleCritical,Luo_inter,Luo_energy_crit}, we establish the following large data scattering and blow-up results which sharpen the ones deduced in \cite{Luo_Waveguide_MassCritical}.

\begin{theorem}[Large data global well-posedness and scattering below ground states]\label{thm large data}
Let $\alpha=\tas$, $m=1$ and let $u$ be a solution of \eqref{nls}. If
\begin{align}\label{gwp condition}
M(u)<2\pi \widehat{M}(Q),\quad E(u)<m_{M(u)},\quad K(u(0))>0,
\end{align}
then $u$ is a global solution of \eqref{nls}. If additionally $d=2$ and $M(u)<\pi \widehat{M}(Q)$, then $u$ also scatters in time in the sense that there exist $\phi^\pm\in H^1(\R^2\times\T)$ such that
\begin{align*}
\lim_{t\to\pm\infty}\|u(t)-e^{it\Delta_{x,y}}\phi^{\pm}\|_{H^1(\R^2\times\T)}=0.
\end{align*}
\end{theorem}

\begin{theorem}[Finite time blow-up below ground states]\label{thm blow up}
Let $\alpha=\tas$, $m=1$ and let $u$ be a solution of \eqref{nls}. If $|x|u(0)\in L^2(\R^d\times\T)$ and
\[M(u)<2\pi \widehat M(Q),\quad \mH(u)<m_{\mM(u)},\quad \mK(u(0))<0,\]
then $u$ blows-up in finite time.
\end{theorem}

\begin{remark}
We should also notice the readers that for the large data scattering result (Theorem \ref{thm large data}) we need to impose the additional assumption that $d=2$ and $M(u)<\pi \widehat{M}(Q)$. This is firstly due to the fact that the large scale bubbles appearing in the linear profile decomposition will approach the large scale focusing mass-critical resonant system (see \cite{Luo_Waveguide_MassCritical,similar_cubic} for details), in which case we need to restrict ourselves to the dimension $d=2$ that guarantees the nonlinearity to be algebraic. Moreover, the mass constraint will ensure the validity of the large data scattering result for the resonant system (Lemma \ref{cnls lem large scale proxy}), which is essential for applications of the stability result (Lemma \ref{lem stability cnls}).
\end{remark}

\begin{remark}
Similarly, we are also able to formulate a large data scattering result for the cubic NLS on $\R^2\times\T^2$ and the quintic NLS on $\R\times\T$ as well as on $\R\times\T^2$, as long as we have the large data scattering results for the corresponding resonant systems. This, however, becomes very technical since either the order of the nonlinear potential or the dimension of the tori increases. As this is out of the scope of the paper, we leave this interesting problem open for future research.
\end{remark}

The rest of the paper is organized as follows: In Section \ref{sec notation} we collect the notation and definitions which will be used throughout the paper. Section \ref{sec 2} to Section \ref{sec 3} are devoted to the proofs of Proposition \ref{prop iden of gamma and beta} to Theorem \ref{thm legendre} (main part of the Legendre-Fenchel identity). In Section \ref{sec 4} to Section \ref{sec 6} we prove Theorem \ref{thm threshold mass 2} to \ref{thm blow up} (applications of the Legendre-Fenchel identity).

\subsubsection*{Acknowledgements}
The author is grateful to Rowan Killip, Changxing Miao, Monica Visan and Jiqiang Zheng for some stimulating discussions. We also want to thank Louis Jeanjean for pointing out a mistake in the first version of the manuscript.

\subsection{Notation and definitions}\label{sec notation}
We use the notation $A\lesssim B$ whenever there exists some positive constant $C$ such that $A\leq CB$. Similarly we define $A\gtrsim B$ and we use $A\sim B$ when $A\lesssim B\lesssim A$.

For simplicity, we ignore in most cases the dependence of the function spaces on their underlying domains and hide this dependence in their indices. For example $L_x^2=L^2(\R^d)$, $H_{x,y}^1= H^1(\R^d\times \T)$
and so on. However, when the space is involved with time, we still display the underlying temporal interval such as $L_t^pL_x^q(I)$, $L_t^\infty L_{x,y}^2(\R)$ etc. The norm $\|\cdot\|_p$ is defined by $\|\cdot\|_p:=\|\cdot\|_{L_{x,y}^p}$.

For $u\in H_{x,y}^1$ and $v\in H_x^1$, define the functionals
\begin{alignat*}{2}
&M(u):=\|u\|^2_{L_{x,y}^2},\qquad&& \widehat{M}(v):=\|v\|_{L_x^2}^2,\\
&E(u):=\frac{1}{2}\|\nabla_{x,y} u\|^2_{L_{x,y}^2}-\frac{1}{\alpha+2}\|u\|^{\alpha+2}_{L_{x,y}^{\alpha+2}},
\qquad&& \widehat{E}(v):=\frac{1}{2}\|\nabla_{x} v\|^2_{L_x^2}-\frac{1}{\alpha+2}\|v\|^{\alpha+2}_{L_{x}^{\alpha+2}},\\
&S_\omega(u):=E(u)+\frac{\omega}{2} M(u),\qquad&& \widehat{S}_\omega(v):=\widehat{E}(u)+\frac{\omega}{2} \widehat{M}(u),\\
&K(u):=\|\nabla_{x} u\|^2_{L_{x,y}^2}-\frac{\alpha d}{2(\alpha+2)}\|u\|^{\alpha+2}_{L_{x,y}^{\alpha+2}},
\qquad&&
\widehat{K}(v):=\|\nabla_{x} v\|^2_{L_{x}^2}-\frac{\alpha d}{2(\alpha+2)}\|v\|^{\alpha+2}_{L_{x}^{\alpha+2}},\\
&N_\omega(u):=\omega M(u)+\|\nabla_{x,y} u\|^2_{L_{x,y}^2}-\|u\|^{\alpha+2}_{L_{x,y}^{\alpha+2}},
\qquad&&
\widehat{N}_\omega(v):=\omega \widehat M(u)+\|\nabla_{x} v\|^2_{L_{x}^2}-\|v\|^{\alpha+2}_{L_{x}^{\alpha+2}}
\end{alignat*}
and
\begin{align}
I_\omega(u)&:=S_\omega(u)-\frac{1}{2}K(u)=\frac{\omega}{2}M(u)+\frac{1}{2}\|\nabla_y u\|_2^2
+ \frac{\alpha d-4}{4(\alpha+2)}\|u\|_{\alpha+2}^{\alpha+2},\label{def of mI beta}\\
I(u)&:=E(u)-\frac{1}{2}K(u)=\frac{1}{2}\|\nabla_y u\|_2^2
+ \frac{\alpha d-4}{4(\alpha+2)}\|u\|_{\alpha+2}^{\alpha+2}.\label{1.4+}
\end{align}
We also define the sets
\begin{alignat*}{2}
S(c)&:=\{u\in H_{x,y}^1:M(u)=c\},\quad
&&V(c):=\{u\in S(c):K(u)=0\},\\
\widehat S(c)&:=\{u\in H_{x}^1:\widehat M(u)=c\},\quad
&&\widehat V(c):=\{u\in \widehat S(c):\widehat K(u)=0\}
\end{alignat*}
on which we study the following variational problems
\begin{align*}
m_c&:=\inf\{E(u):u\in V(c)\},\\
\gamma_\omega&:=\inf\{S_\omega(u):u\in H_{x,y}^1\setminus\{0\},\,K(u)=0\},\\
\beta_\omega&:=\inf\{S_\omega(u):u\in H_{x,y}^1\setminus\{0\},\,N_\omega(u)=0\},\\
\widehat m_c&:=\inf\{\widehat E(u):u\in \widehat V(c)\},\\
\widehat\gamma_\omega&:=\inf\{\widehat S_\omega(u):u\in H_{x}^1\setminus\{0\},\,\widehat K(u)=0\},\\
\widehat\beta_\omega&:=\inf\{\widehat S_\omega(u):u\in H_{x}^1\setminus\{0\},\,\widehat N_\omega(u)=0\}.
\end{align*}
Finally, we introduce the concept of an \textit{admissible} pair on $\R^d$. A pair $(q,r)$ is said to be $\dot{H}^s$-admissible if $q,r\in[2,\infty]$, $s\in[0,\frac d2)$, $\frac{2}{q}+\frac{d}{r}=\frac{d}{2}-s$ and $(q,d)\neq(2,2)$.

\section{Existence and $y$-dependence of the ground states with fixed frequency: Proof of Theorem \ref{thm existence ground states 1} and \ref{thm threshold frequency}}\label{sec 2}

\subsection{Existence of ground states}
We begin with the proof of Theorem \ref{thm existence ground states 1}. The proof will make use of several auxiliary lemmas, which we state in the following.

As a first auxiliary result, we record the following useful concentration compactness lemma which will help us to find a non-vanishing weak limit of a minimizing sequence.

\begin{lemma}[Non-vanishing weak limit, \cite{TTVproduct2014}]\label{lemma non vanishing limit}
Let $(u_n)_n$ be a bounded sequence in $H_{x,y}^1$. Assume also that there exists some $\alpha\in(0,\tbs)$ such that
\begin{align*}
\liminf_{n\to\infty}\|u_n\|_{\alpha+2}>0.
\end{align*}
Then there exists $(x_n)_n\subset \R^d$ and some $u\in H_{x,y}^1\setminus\{0\}$ such that, up to a subsequence,
\begin{align*}
u_n(\cdot+x_n)\rightharpoonup u\quad\text{weakly in $H_{x,y}^1$}.
\end{align*}
\end{lemma}

We also record a useful scale-invariant Gagliardo-Nirenberg inequality on $\R^d\times\T^m$, which was originally proved in \cite{Luo_inter} in the case $m=1$. As the generalization to any $\alpha,d,m$ satisfying \eqref{add} is straightforward, we omit the details here.

\begin{lemma}[Scale-invariant Gagliardo-Nirenberg inequality on $\R^d\times\T^m$, \cite{Luo_inter}]\label{lemma gn additive}
Assume that $\alpha,d,m$ satisfy \eqref{add}. Then there exists some $C>0$ such that for all $u\in H_{x,y}^1$ we have
\begin{align*}
\|u\|_{\alpha+2}^{\alpha+2}\leq C\|\nabla_x u\|_2^{\frac{\alpha d}{2}}\|u\|_2^{\frac{4-\alpha(d+m-2)}{2}}
(\| u\|_{2}^{\frac{\alpha m}{2}}+\|\nabla_y u\|_{2}^{\frac{\alpha m}{2}}).
\end{align*}
\end{lemma}

The following corollary is an immediate consequence of Lemma \ref{lemma gn additive}.

\begin{corollary}\label{cor lower bound}
For any $\omega\in(0,\infty)$ we have $\gamma_\omega\in(0,\infty)$. Moreover, there exists a minimizing sequence $(u_n)_n$ of $\gamma_\omega$ such that $(u_n)_n$ is bounded in $H_{x,y}^1$ and $\liminf_{n\to\infty}\|u_n\|_{\alpha+2}>0$.
\end{corollary}

\begin{proof}
That $\gamma_\omega<\infty$ follows from the fact that there exists some $u\in H_{x,y}^1\setminus\{0\}$ with $K(u)=0$. Indeed, by direct computation one immediately verifies that the mapping $t\mapsto K(tu)$ is positive for $0<t\ll 1$ and $\lim_{t\to\infty} K(tu)=-\infty$, from which the former claim follows. Next, let $(u_n)_n$ be a minimizing sequence of $\gamma_\omega$. Then
\begin{align}\label{uniform bddnes}
\infty>\gamma_\omega+o_n(1)&=\mN_\omega(u_n)=\mN_\omega(u_n)-\frac{2}{\alpha d}\mK(u_n)\nonumber\\
&=\frac{\omega}{2}\mM(u)+\frac{1}{2}\|\nabla_y u_n\|_2^2+\bg(\frac{1}{2}-\frac{2}{\alpha d}\bg)\|\nabla_x u_n\|_2^2.
\end{align}
In the case $\alpha>4/d$ (which in turn implies $\frac{1}{2}-\frac{2}{\alpha d}>0$) we infer that $(u_n)_n$ is a bounded sequence in $H_{x,y}^1$. In the case $\alpha=4/d$ the energies $S_\omega(u)$ and $K(u)$ are stable w.r.t. replacing $u$ to $v^t$ defined by \eqref{1.4} for any $t>0$. Thus we may w.l.o.g. assume that $\|\nabla_x u_n\|_2^2=\frac{1}{2}\|u_n\|_{\tas+2}^{\tas+2}=1$, from which the boundedness of $(u_n)_n$ in $H_{x,y}^1$ and $\liminf_{n\to\infty}\|u_n\|_{\tas+2}>0$ follow. Now by Lemma \ref{lemma gn additive} and the fact that $\alpha<\tbs$ we deduce
\begin{align}\label{key gn inq}
\|\nabla_xu_n\|_2^2&=\frac{\alpha d}{2(\alpha+2)}\|u_n\|_{\alpha+2}^{\alpha+2}\nonumber\\
&\lesssim \|\nabla_x u_n\|_2^{\frac{\alpha d}{2}}\|u_n\|_2^{\frac{4-\alpha(d+m-2)}{2}}
(\| u_n\|_{2}^{\frac{\alpha m}{2}}+\|\nabla_y u_n\|_{2}^{\frac{\alpha m}{2}})\lesssim \|\nabla_x u_n\|_2^{\frac{\alpha d}{2}}.
\end{align}
In the case $\alpha>4/d$, \eqref{key gn inq} implies
\begin{align}
\liminf_{n\to\infty}\|u_n\|_{\alpha+2}^{\alpha+2}\sim\liminf_{n\to\infty}\|\nabla_x u_n\|_2^2>0.\label{key gn inq2}
\end{align}
Hence
\begin{align}
\gamma_\omega&=\lim_{n\to\infty}\bg(\frac{\omega}{2}\mM(u)+\frac{1}{2}\|\nabla_y u_n\|_2^2+\bg(\frac{1}{2}-\frac{2}{\alpha d}\bg)\|\nabla_x u_n\|_2^2\bg)\nonumber\\
&\geq \bg(\frac{1}{2}-\frac{2}{\alpha d}\bg)\liminf_{n\to\infty}\|\nabla_x u_n\|_2^2\gtrsim 1.\label{2.5}
\end{align}
In the case $\alpha=4/d$, if $M(u_n)\neq o_n(1)$, then
\[\gamma_\omega\geq \liminf_{n\to\infty}\frac{\omega}{2}M(u_n)>0.\]
Otherwise assume that $M(u_n)=o_n(1)$. Then from the computations given in \eqref{key gn inq} we obtain
\[1\leq  C o_n(1)(o_n(1)+\|\nabla_y u_n\|_{2}^{\frac{\alpha m}{2}})\]
which in turn implies
\[\|\nabla_y u_n\|_{2}^{\frac{\alpha m}{2}}\geq \frac{1}{Co_n(1)}-o_n(1)\to \infty\]
as $n\to\infty$. We thus conclude that $(\|\nabla_y u_n\|_{2})_n$ is unbounded, which contradicts the $H_{x,y}^1$-boundedness of $(u_n)_n$. This completes the desired proof.
\end{proof}

\begin{lemma}[Pohozaev's identity for \eqref{nls2} on $\R^d\times\T^m$]\label{lemma of poho}
Let $u\in H_{x,y}^1$ be a solution of \eqref{nls2} with some $\omega>0$. Then $K(u)=0$.
\end{lemma}

\begin{proof}
Using standard elliptic regularity theory (see for instance \cite[Lem. B.3]{Struwe1996}) and bootstrap arguments (see for instance \cite[Thm. 8.1.1]{Cazenave2003}) we know that $u\in W_{\rm loc}^{3,q}(\R^{d+m})$ for all $q\in[1,\infty)$. Direct computation yields
\begin{align}
\mathrm{Re}(-\Delta_x u(x\cdot\nabla_x \bar{u}))&=\nabla_x\cdot
(-\mathrm{Re}(\nabla_x u(x\cdot\nabla_x \bar{u}))+x\frac{|\nabla_x u|^2}{2})-\frac{d-2}{2}|\nabla_x u|^2,\label{pohozaev1}\\
\mathrm{Re}((\omega u-|u|^\alpha u)x\cdot\nabla_x\bar{u})&=\mathrm{div}_x(x(\frac{\omega}{2}|u|^2-\frac{1}{\alpha+2}|u|^{\alpha+2}))
-d(\frac{\omega}{2}|u|^2-\frac{1}{\alpha+2}|u|^{\alpha+2}),\label{pohozaev2}\\
\mathrm{Re}(-\Delta_y u (x\cdot\nabla_x \bar{u}))&=\mathrm{Re}\left(-\mathrm{div}_y(\nabla_y u (x\cdot\nabla_x u))\right)
+\nabla_x\cdot\left(x\frac{|\nabla_y u|^2}{2}\right)-\frac{d|\nabla_y u|^2}{2}.\label{pohozaev3}
\end{align}
Define
\begin{align*}
R_n:=B_n^x\times\T^m,
\end{align*}
where $B_n^x=\{x\in\R^d:|x|<n\}$. Integrate \eqref{pohozaev1} to \eqref{pohozaev3} over $R_n$ and followed by divergence theorem and the periodic boundary conditions, we obtain
\begin{align}
0=&-\bg(\int_{\T^m}\int_{\pt B_n^x}\mathrm{Re}((\nabla_x u \cdot n_x)(\sigma_x\cdot\nabla_x \bar{u}))-(\sigma_x\cdot n_x)\frac{|\nabla_x u|^2}{2}
\,d\sigma_xdy\bg)\label{pohoa}\\
&+\int_{\T^m}\int_{\pt B_n^x}(\sigma_x\cdot n_x)\left(\frac{\omega}{2}|u|^2-\frac{1}{\alpha+2}|u|^{\alpha+2}\right)
\,d\sigma_xdy\label{pohob}\\
&+\int_{\T^m}\int_{\pt B_n^x}(\sigma_x\cdot n_x)\frac{|\nabla_y u|^2}{2}
\,d\sigma_xdy\label{pohoc}\\
&-\frac{d-2}{2}\|\nabla_x u\|_{L^2(R_n)}^2-\frac{d}{2}\|\nabla_y u\|_{L^2(R_n)}^2-\frac{d\omega}{2}\|u\|_{L^2(R_n)}^2+\frac{d}{\alpha+2}\|u\|_{L^{\alpha+2}(R_n)}^{\alpha+2}.
\end{align}
Particularly, the terms \eqref{pohoa} to \eqref{pohoc} are well-defined, thanks to the smoothness of $u$. Arguing as in the proof of \cite[Prop. 1]{lions1}, there exists a sequence $(n)_n\subset(0,\infty)$ with $n\to\infty$ such that
$$ \eqref{pohoa}+\eqref{pohob}+\eqref{pohoc}\to 0$$
as $n\to\infty$. Combining with the dominated convergence theorem, we infer that
\begin{align}\label{1poho}
0=\frac{d-2}{2}\|\nabla_x u\|_{2}^2+\frac{d}{2}\|\nabla_y u\|_{2}^2+\frac{d\omega}{2}\|u\|_{2}^2-\frac{d}{\alpha+2}\|u\|_{\alpha+2}^{\alpha+2}.
\end{align}
On the other hand, multiplying \eqref{nls2} with $\bar{u}$ and integrating over $\R^d\times\T^m$, we obtain
\begin{align}\label{2poho}
0=\|\nabla_x u\|_{2}^2+\|\nabla_y u\|_{2}^2+\omega\|u\|_{2}^2-\|u\|_{\alpha+2}^{\alpha+2}.
\end{align}
Eliminating $\|\nabla_y u\|_2^2+\omega\|u\|_2^2$ in \eqref{1poho} and \eqref{2poho} we conclude that $K(u)=0$.
\end{proof}

\begin{lemma}[Optimizer as a standing wave solution]\label{vc implies sc}
For $\omega\in(0,\infty)$ let $u$ be an optimizer of $\gamma_\omega$. Then there exists some $t>0$ such that $v(x):=t^{\frac{d}{2}}u(tx,y)$ solves \eqref{nls2} with the same given $\omega$.
\end{lemma}
\begin{proof}
By Lagrange multiplier theorem there exists $\mu \in\R$ such that
\begin{align*}
S_\omega'(u)-\mu K'(u)=0,
\end{align*}
or equivalently
\begin{align}\label{pde on Vc}
(1-2\mu)(-\Delta_x u)+(-\Delta_y u)+\omega u+(\frac{\mu\alpha d}{2}-1)|u|^{\alpha}u=0.
\end{align}
If $\mu\geq\frac12$, then using $K(u)=0$ we obtain
\begin{align*}
(1-2\mu)\|\nabla_x u\|_2^2+(\frac{\mu\alpha d}{2}-1)\|u\|_{\alpha+2}^{\alpha+2}
&=\frac{\alpha d-2(\alpha+2)+\mu\alpha^2 d}{2(\alpha+2)}\|u\|_{\alpha+2}^{\alpha+2}\nonumber\\
&\geq(\alpha d-2(\alpha+2)+\frac{1}{2}\alpha^2 d)(2(\alpha+2))^{-1}\|u\|_{\alpha+2}^{\alpha+2}.
\end{align*}
Calculating the derivative of $f(\alpha):=\alpha d-2(\alpha+2)+\frac{1}{2}\alpha^2 d$ and using $\alpha\geq 4/d$ we obtain
\begin{align*}
f'(\alpha)=d-2+\alpha d\geq d+2>0.
\end{align*}
Therefore
$$\alpha d-2(\alpha+2)+\frac{1}{2}\alpha^2 d\geq f(4/d)=0.$$
Consequently, testing \eqref{pde on Vc} with $\bar u$ yields
\begin{align}\label{lower bound}
0\geq\|\nabla_y u\|_2^2+\omega\|u\|_2^2,
\end{align}
which implies the contradiction $u=0$, and thus $\mu< \frac{1}{2}$. In this case, by standard elliptic regularity and bootstrap arguments we know that $u\in W^{3,q}_{\rm loc}(\R^{d+1})$ for all $q\in[2,\infty)$, hence using the arguments in Lemma \ref{lemma of poho} we obtain the following Pohozaev's identity corresponding to \eqref{pde on Vc}:
\begin{align}
0=(1-2\mu)\|\nabla_x u\|_2^2+\bg(\frac{\mu\alpha d}{2}-1\bg)\bg(\frac{\alpha d}{2(\alpha+2)}\bg)\|u\|_{\alpha+2}^{\alpha+2}.\label{2.17}
\end{align}
In the case $\alpha>4/d$, together with $K(u)=0$ it follows by eliminating $\|\nabla_x u\|_2^2$ that
\begin{align*}
\mu\|u\|_{\alpha+2}^{\alpha+2}=0,
\end{align*}
which in turn implies $\mu=0$ since $u\neq 0$. The claim thus follows by choosing $t=1$ and $v=u$. In the case $\alpha=4/d$ we obtain $\mu\alpha d/2-1=2\mu-1$. Thus we may simply choose $t=\sqrt{1-2\mu}$ and the claim follows from \eqref{pde on Vc}. This completes the proof.
\end{proof}

\begin{lemma}[Equivalent formulation for $\gamma_\omega$]\label{identification m gamma}
For any $\omega\in(0,\infty)$ define
\begin{align}
\bar \gamma_\omega:=\inf\{\mI_\omega(u):u\in H_{x,y}^1\setminus\{0\},\mK(u)\leq 0\},
\end{align}
where $\mI_\omega(u)$ is defined by \eqref{def of mI beta}. Then $\gamma_\omega=\bar \gamma_\omega$.
\end{lemma}

\begin{proof}
Let $(u_n)_n\subset H_{x,y}^1\setminus\{0\}$ be a minimizing sequence for the variational problem of $\bar\gamma_\omega$, i.e.
\begin{align*}
\mI_\omega(u_n)=\bar\gamma_\omega+o_n(1),\quad
\mK(u_n)\leq 0\quad\forall\,n\in\N.
\end{align*}
By fundamental computation we know that there exists some $t_n\in(0,1]$ such that $\mK(t_n u_n)$ is equal to zero. Thus
\begin{align*}
\gamma_\omega&\leq S_\omega(t_n u_n)=\mI_\omega(t_n u_n)+\frac12 K_\omega(t_n u_n)=\mI_\omega(t_n u_n)\nonumber\\
&=t_n^2\bg(\frac{\omega}{2}M(u_n)+\frac{1}{2}\|\nabla_y u_n\|_2^2\bg)
+ \frac{t_n^{\alpha+2}(\alpha d-4)}{4(\alpha+2)}\|u_n\|_{\alpha+2}^{\alpha+2}\nonumber\\
&\leq
\frac{\omega}{2}M(u_n)+\frac{1}{2}\|\nabla_y u_n\|_2^2
+ \frac{\alpha d-4}{4(\alpha+2)}\|u_n\|_{\alpha+2}^{\alpha+2}
=\bar \gamma_\omega+o_n(1).
\end{align*}
Sending $n\to\infty$ we infer that $\gamma_\omega\leq \bar\gamma_\omega$. On the other hand, by definition
\begin{align*}
\bar\gamma_\omega
&\leq\inf\{\mI_\omega(u):u\in H_{x,y}^1\setminus\{0\},\,K(u)=0\}\nonumber\\
&=\inf\{\mN_\omega(u):u\in H_{x,y}^1\setminus\{0\},\,K(u)=0\}=\gamma_{\omega}.
\end{align*}
This completes the desired proof.
\end{proof}

Having all the preliminaries we are in a position to prove Theorem \ref{thm existence ground states 1}.

\begin{proof}[Proof of Theorem \ref{thm existence ground states 1}]
Let $(u_n)_n\subset H_{x,y}^1\setminus\{0\}$ be a minimizing sequence of $\bar\gamma_\omega$, i.e.
\begin{align*}
\mI_\omega(u_n)=\bar\gamma_\omega+o_n(1),\quad \mK(u_n)\leq 0.
\end{align*}
By diamagnetic inequality we know that $\bar\gamma_\omega$ is stable under the mapping $u\mapsto |u|$, thus we may w.l.o.g. assume that $u_n\geq0$. Arguing as in Corollary \ref{cor lower bound} we infer that $(u_n)_n$ is a bounded sequence in $H_{x,y}^1$ with
\begin{align*}
\liminf_{n\to\infty}\|u_n\|_{\alpha+2}\gtrsim 1.
\end{align*}
By Lemma \ref{lemma non vanishing limit} we may find a sequence $(x_n)_n\subset\R^d$ (which by spatial translation invariance can be w.l.o.g. assumed equal zero) and some $0\leq u_\omega\in S$ such that (up to a subsequence)
\begin{align*}
u_n \rightharpoonup u_\omega\quad\text{in $H_{x,y}^1$}.
\end{align*}
We now show that $u_\omega$ is an optimizer of $\gamma_\omega$. By weakly lower semicontinuity of norms and Lemma \ref{identification m gamma} we already know that $\mI_\omega(u_\omega)\leq \bar\gamma_\omega=\gamma_\omega$. It remains to prove that $\mK(u_\omega)= 0$. Assume first $\mK(u_\omega)>0$. Using the Brezis-Lieb lemma we have
\begin{align*}
\mK(u_n-u_\omega)+\mK(u_\omega)=\mK(u_n)+o_n(1).
\end{align*}
Combining with the fact that $\mK(u_n)\leq 0$ we know that $\mK(u_n-u_\omega)<0$ for all sufficiently large $n$. Thus we can find some $t_n\in(0,1)$ such that $\mK(t_n(u_n-u_\omega))=0$. Applying the Brezis-Lieb lemma again we infer that
\begin{align*}
\bar\gamma_\omega\leq \mI_\omega(t_n(u_n-u_\omega))\leq \mI_\omega(u_n-u_\omega)=\mI_\omega(u_n)-\mI_\omega(u_\omega)+o_n(1)=\bar\gamma_\omega-\mI_\omega(u_\omega)+o_n(1).
\end{align*}
Combining with the non-negativity of $\mI_\omega(u_\omega)$ we infer that $\mI_\omega(u_\omega)=0$, which in turn implies $u_\omega=0$, a contradiction. Next, if $\mK(u_\omega)<0$, then we can find some $t\in(0,1)$ such that $\mK(tu_\omega)=0$ and consequently
\begin{align*}
\bar \gamma_\omega\leq \mI_\omega(t u_\omega)<\mI_\omega(u_\omega)\leq \bar\gamma_\omega
\end{align*}
which is again a contradiction. Hence $\mK(u_\omega)=0$ and $u_\omega$ is an optimizer of $\gamma_{\omega}$. That $u_\omega$ is a solution (up to symmetries) of \eqref{nls2} follows from Lemma \ref{vc implies sc} and the positivity of $u_\omega$ follows from the strong maximum principle. Finally, in the case $p=\tas$, that the function $v^t$ defined by \eqref{1.4} is also an optimizer of $\gamma_\omega$ follows from the fact that the energies $S_\omega(u)$ and $K(u)$ are stable under the transformation $u\mapsto v^t$ for any $t>0$. This completes the proof.
\end{proof}

\subsection{$y$-dependence of the ground states}
This subsection is devoted to the proof of Theorem \ref{thm threshold frequency}. Many of the arguments are similar to the ones given in \cite{Luo_inter}, but for the sake of completeness we will present the full details of the proof.

For $\omega>0$ define
\begin{gather*}
E_{1,\omega}(u):=\frac{1}{2}\|\nabla_{x} u\|^2_{2}+\frac{1}{2\sqrt{\omega}}\|\nabla_y u\|_2^2-\frac{1}{\alpha+2}\|u\|^{\alpha+2}_{L_{x,y}^{\alpha+2}},\\
S_{1,\omega}(u):=E_{1,\omega}(u)+\frac{1}{2} M(u),\\
\gamma_{1,\omega}:=\inf\{S_{1,\omega}(u):u\in H_{x,y}^1\setminus\{0\},\,K(u)=0\}.
\end{gather*}
Particularly, arguing as in the proof of Theorem \ref{thm existence ground states 1} we know that for any $\omega>0$ the variational problem $\gamma_{1,\omega}$ has at least a positive optimizer $u_{1,\omega}$.

Our first goal is to prove the following characterization of $\gamma_{1,\omega}$ for varying $\omega$.

\begin{lemma}\label{lemma auxiliary}
There exists some $\omega_*\in(0,\infty)$ such that
\begin{itemize}
\item For all $\omega\in(0,\omega_*)$ we have $\gamma_{1,\omega}=(2\pi)^m \widehat \gamma_{1}$. Moreover, any minimizer $u_{1,\omega}$ of $\gamma_{1,\omega}$ satisfies $\nabla_y u_{1,\omega}= 0$.
\item For all $\omega\in(\omega_*,\infty)$ we have $\gamma_{1,\omega}<(2\pi)^m \widehat \gamma_{1}$. Moreover, any minimizer $u_{1,\omega}$ of $\gamma_{1,\omega}$ satisfies $\nabla_y u_{1,\omega}\neq 0$.
\end{itemize}
\end{lemma}

In order to prove Lemma \ref{lemma auxiliary}, we firstly collect some useful auxiliary lemmas.

\begin{lemma}\label{auxiliary lemma 1}
We have
\begin{align}\label{limit ld to infty energy sec4}
\lim_{\omega\to 0}\gamma_{1,\omega}=(2\pi)^m\wm_1.
\end{align} Moreover, if $u_{1,\omega}$ is an optimizer of $\gamma_{1,\omega}$, then
\begin{align}
\lim_{\omega\to 0}\omega^{-\frac12}\|\nabla_y u_{1,\omega}\|_2^2=0.\label{vanishing sec4}
\end{align}
\end{lemma}

\begin{proof}
By assuming that a candidate satisfying $K(u)=0$ is independent of $y$ we already conclude that
\begin{align}
\gamma_{1,\omega}\leq (2\pi)^m \wm_1.\label{upper bound sec4}
\end{align} Next we prove
\begin{align}
\lim_{\omega\to 0}\|\nabla_y u_{1,\omega}\|_2^2=0.\label{vanishing1 sec4}
\end{align}
Suppose that \eqref{vanishing1 sec4} does not hold. Then we must have
\begin{align*}
\lim_{\omega\to 0}\omega^{-\frac12}\|\nabla_y u_{1,\omega}\|_2^2=\infty.
\end{align*}
Since $\mK(u_{1,\omega})=0$ and $\alpha\geq 4/d$,
\begin{align}
\gamma_{1,\omega}&=\mN_{1,\omega}(u_{1,\omega})-\frac{2}{\alpha d}\mK_{\omega}(u_{1,\omega})\nonumber\\
&=\frac{1}{2}\mM(u_{1,\omega})+\frac{1}{2\sqrt{\omega}}\|\nabla_y u_{1,\omega}\|_2^2
+(\frac12-\frac{2}{\alpha d})\|\nabla_x u_{1,\omega}\|_2^2
\geq \frac{1}{2\sqrt{\omega}}\|\nabla_y u_{1,\omega}\|_2^2\to\infty\label{contradiction sec4}
\end{align}
as $\omega\to 0$, which contradicts \eqref{upper bound sec4} and in turn proves \eqref{vanishing1 sec4}. Using \eqref{upper bound sec4} and \eqref{contradiction sec4} we infer that
\begin{align}\label{upper bound 2 sec4}
\mM(u_{1,\omega})+\|\nabla_x u_{1,\omega}\|_2^2\lesssim  \gamma_{1,\omega}\leq (2\pi)^m \wm_1<\infty.
\end{align}
Therefore $(u_{1,\omega})_\omega$ is a bounded sequence in $H_{x,y}^1$, whose weak limit is denoted by $u$. By \eqref{vanishing1 sec4} we know that $u$ is independent of $y$ and thus $u\in H_x^1$. Using Corollary \ref{cor lower bound} and Lemma \ref{lemma non vanishing limit} we also infer that $u\neq 0$. On the other hand, by Lemma \ref{vc implies sc} we know that (up to scaling in $x$-direction)
\begin{align}
-\Delta_x u_{1,\omega}-\omega^{-\frac12} \Delta_y u_{1,\omega}+u_{1,\omega}=|u_{1,\omega}|^\alpha u_{1,\omega}\quad\text{on $\R^d\times \T^m$}.\label{vanishing 3 sec4}
\end{align}
We now test \eqref{vanishing 3 sec4} with $\phi\in C_c^\infty(\R^d)$ and integrate both sides over $\R^d\times\T^m$. Notice particularly that the term $\int_{\R^d\times\T^m}  \nabla_y u_{1,\omega} \nabla_y\phi\,dxdy=0$ for any $\omega>0$ since $\phi$ is independent of $y$. Using then the weak convergence of $u_{1,\omega}$ to $u$, by sending $\omega\to 0$ we obtain
\begin{align}
-\Delta_x u+u=|u|^\alpha u\quad\text{on $\R^d$}.\label{vanishing 4 sec4}
\end{align}
In particular, by the Pohozaev's identity on $\R^d$ we know that $\wmK(u)=0$ and consequently $\wmN_1(u)\geq \wm_1$. Finally, using the weakly lower semicontinuity of norms we conclude
\begin{align}
\gamma_{1,\omega}&=\mN_{1,\omega}(u_{1,\omega})=\mN_{1,\omega}(u_{1,\omega})-\frac{2}{\alpha d}\mK(u)\nonumber\\
&=\frac{1}{2\sqrt{\omega}}\|\nabla_y u_{1,\omega}\|_2^2+\frac{1}{2}\mM(u_{1,\omega})+(\frac{1}{2}-\frac{2}{\alpha d})\|\nabla_x u_{1,\omega}\|_2^2\nonumber\\
&\geq \frac{1}{2}\mM(u_{1,\omega})+(\frac{1}{2}-\frac{2}{\alpha d})\|\nabla_x u_{1,\omega}\|_2^2\nonumber\\
&\geq \frac{1}{2}\mM(u)+(\frac{1}{2}-\frac{2}{\alpha d})\|\nabla_x u\|_2^2+o_\omega(1)\nonumber\\
&=(2\pi)^m \wmN_1(u)+o_\omega(1)\geq (2\pi)^m\wm_1+o_\omega(1).\label{vanishing 5 sec4}
\end{align}
Letting $\omega\to 0$ and taking \eqref{upper bound sec4} into account we conclude \eqref{limit ld to infty energy sec4}. Finally, \eqref{vanishing sec4} follows directly from the previous computations by not neglecting $\omega^{-\frac12}\|u_{1,\omega}\|_2^2$ therein. This completes the desired proof.
\end{proof}

\begin{lemma}\label{strong convergence u ld}
There exists some $u\in H_x^1\setminus\{0\}$ such that up to a subsequence, $u_{1,\omega}\to u$ strongly in $H_{x,y}^1$.
\end{lemma}

\begin{proof}
We first consider the case $\alpha\neq \tas$. From the proof of Lemma \ref{auxiliary lemma 1} we know that there exists some $u\in H_x^1\setminus\{0\}$ such that $u_{1,\omega}\rightharpoonup u$ weakly in $H_{x,y}^1$ and $u$ is an optimizer of $\wm_1$. Using $\mK(u_{1,\omega})=\wmK(u)=0$, weakly lower semicontinuity of norms, \eqref{vanishing sec4} and \eqref{limit ld to infty energy sec4} we obtain
\begin{align*}
(2\pi)^m\wm_1&=(2\pi)^m(\wmN_1(u)-\frac{2}{\alpha d}\wmK(u))=(2\pi)^m\bg(\frac{1}{2}\wmM(u)+(\frac{1}{2}-\frac{2}{\alpha d})\|\nabla_x u\|_{L_x^2}^2\bg)\nonumber\\
&\leq \liminf_{\omega\to 0}\bg(\frac{1}{2}\mM(u_{1,\omega})+(\frac{1}{2}-\frac{2}{\alpha d})\|\nabla_x u_{1,\omega}\|_{2}^2+\frac{1}{2\sqrt{\omega}}\|\nabla_y u_{1,\omega}\|_2^2\bg)
= \lim_{\omega\to 0}\gamma_{1,\omega}=(2\pi)^m\wm_1.
\end{align*}
The equality holds if and only if
\begin{align*}
\|u_{1,\omega}\|_2^2\to (2\pi)^m\|u\|_{L_x^2}^2\quad\text{and}\quad \|\nabla_x u_{1,\omega}\|_2^2\to(2\pi)^m\|\nabla_x u\|_{L_x^2}^2
\end{align*}
as $\omega\to 0$. Combining with \eqref{vanishing1 sec4} and the fact that $u_{1,\omega}\rightharpoonup u$ weakly in $H_{x,y}^1$ we conclude the strong convergence of $u_{1,\omega}$ to $u$ in $H_{x,y}^1$ as $\omega\to 0$.

We now consider the case $\alpha=\tas$. In this case we are unable to obtain $\|\nabla_x u_{1,\omega}\|_2^2\to(2\pi)^m\|\nabla_x u\|_{L_x^2}^2$ using the previous arguments since $\frac{1}{2}-\frac{2}{\alpha d}=0$. Nevertheless, it still holds $\lim_{\omega\to 0}\gamma_{1,\omega}=(2\pi)^m\wm_1$. By defining
\[N_{1,\omega}(u):=\|\nabla_x u\|_2^2
+\omega^{-\frac12}\|\nabla_y u\|_2^2+M(u)-\|u\|_{\tas+2}^{\tas+2}\]
we also know that $N_{1,\omega}(u_{1,\omega})=0$. From this and $K(u_{1,\omega})=0$ it follows $\gamma_{1,\omega}=\frac{\tas}{2(\tas+2)}\|u_{1,\omega}\|_{\tas+2}^{\tas+2}=d^{-1}\|\nabla_x u_{1,\omega}\|_2^2$,
see the last part of the proof of Theorem \ref{thm existence ground states 1}. In the same manner we deduce $\wm_1=d^{-1}\|\nabla_x u\|_{L_x^2}^2$. Summing up the arguments we conclude again that $\|\nabla_x u_{1,\omega}\|_2^2\to(2\pi)^m\|\nabla_x u\|_{L_x^2}^2$, which completes the desired proof.
\end{proof}

\begin{lemma}\label{lemma no dependence}
There exists some $\omega_0$ such that $\nabla_y u_{1,\omega}=0$ for all $\omega<\omega_0$.
\end{lemma}

\begin{proof}
Let $w_\omega:=\nabla_y u_{1,\omega}$. Then using Lemma \ref{vc implies sc} we obtain
\begin{align}\label{no dependence 1}
-\Delta_x w_\omega+\omega^{-\frac12} w_\omega+w_\omega=\nabla_y(|u_{1,\omega}|^\alpha u_{1,\omega})=(\alpha+1)|u_{1,\omega}|^\alpha w_\omega.
\end{align}
Testing \eqref{no dependence 1} with $w_\omega$ and rewriting suitably, we infer that
\begin{align}
0&=\|\nabla_x w_\omega\|_2^2+\omega^{-\frac12}\|\nabla_y w_\omega\|_2^2+\|w_\omega\|_2^2-(\alpha+1)\int_{\R^d\times \T^m}|u_{1,\omega}|^\alpha |\bar{w}_\omega|^2\,dxdy\nonumber\\
&=(\omega^{-\frac12}-1)\|\nabla_y w_\omega\|_2^2-(\alpha+1)\int_{\R^d\times \T^m}|u|^\alpha |w_\omega|^2\,dxdy\label{no dependence 2}\\
&+\|w_\omega\|_{H_{x,y}^1}^2-(\alpha+1)\int_{\R^d\times \T^m}(|u_{1,\omega}|^\alpha -|u|^\alpha)|w_\omega|^2\,dxdy.\label{no dependence 3}
\end{align}
For \eqref{no dependence 2}, we firstly point out that by the classical elliptic regularity on $\R^d$ and the Sobolev's embedding we have $u\in L^\infty(\R^d)$. On the other hand, since $\int_\T w_\omega\,dy=0$, we have $\|w_\omega\|_2\leq \|\nabla_y w_\omega\|_2$. Summing up, we conclude that
\begin{align*}
\eqref{no dependence 2}\geq (\omega^{-\frac12}-1-(\alpha+1)\|u\|^{\alpha}_{L_x^\infty})\|\nabla_y w_\omega\|_2^2\geq 0
\end{align*}
for all $0<\omega\ll 1$. For \eqref{no dependence 3}, we discuss the cases $\alpha\leq 1$ and $\alpha>1$ separately. For $\alpha\leq 1$, we estimate the second term in \eqref{no dependence 3} using subadditivity of concave function, H\"older's inequality, Lemma \ref{strong convergence u ld} and the Sobolev embedding $H_{x,y}^1\hookrightarrow L_{x,y}^{\alpha+2}$:
\begin{align*}
&\,\int_{\R^d\times \T^m}(|u_{1,\omega}|^\alpha -|u|^\alpha)|w_\omega|^2\,dxdy\nonumber\\
\leq&\,\int_{\R^d\times \T^m}|u_{1,\omega}-u|^\alpha|w_\omega|^2\,dxdy\nonumber\\
\leq& \,\|u_{1,\omega}-u\|_{\alpha+2}^\alpha\|w_\omega\|_{\alpha+2}^2\leq o_\omega(1)\|w_\omega\|^2_{H_{x,y}^1}.
\end{align*}
The case $\alpha>1$ can be similarly estimated as follows:
\begin{align*}
&\,\int_{\R^d\times \T^m}(|u_{1,\omega}|^\alpha -|u|^\alpha)|w_\omega|^2\,dxdy\nonumber\\
\lesssim&\,\int_{\R^d\times \T^m}|u_{1,\omega}-u||w_\omega|^2(|u_{1,\omega}|^{\alpha-1}+|u|^{\alpha-1})\,dxdy\nonumber\\
\leq& \,\|u_{1,\omega}-u\|_{\alpha+2}(\|u_{1,\omega}\|^{\alpha-1}_{\alpha+2}+\|u\|^{\alpha-1}_{\alpha+2})\|w_\omega\|_{\alpha+2}^2\leq o_\omega(1)\|w_\omega\|^2_{H_{x,y}^1}.
\end{align*}
Therefore, \eqref{no dependence 2} and \eqref{no dependence 3} imply
\begin{align*}
0\geq \|w_\omega\|_{H_{x,y}^1}^2(1-o_\omega(1))\gtrsim \|w_\omega\|_{H_{x,y}^1}^2
\end{align*}
for all $\omega<\omega_0$ for some sufficiently large $\omega_0$. We therefore conclude that $0=w_\omega=\nabla_y u_{1,\omega}$ for all $\omega>\omega_0$ and the desired proof is complete.
\end{proof}

Having all the preliminaries we are in a position to prove Lemma \ref{lemma auxiliary}.

\begin{proof}[Proof of Lemma \ref{lemma auxiliary}]
Define
\begin{align*}
\omega_*:=\sup\{\omega\in(0,\infty):\gamma_{1,\omega}=(2\pi)^m \wm_{1}\text{ for all $\omega\leq\omega$}\}.
\end{align*}
From Lemma \ref{lemma no dependence} it already follows that $\omega_*<\infty$ and we shall need to show $\omega_*>0$. If suffices to show
\begin{align*}
\lim_{\omega\to \infty}\gamma_{1,\omega}<(2\pi)^m\wm_{1}.
\end{align*}
To see this, we firstly define the function $\rho:[0,2\pi]\to[0,\infty)$ as follows: Let $a\in(0,\pi)$ such that $a>\pi-3\pi\bg(\frac{3}{\alpha+3}\bg)^{\frac2\alpha}$. This is always possible for $a$ sufficiently close to $\pi$. Then we define $\rho$ by
\begin{align*}
\rho(y)=\left\{
\begin{array}{ll}
0,&y\in[0,a]\cup[2\pi-a,2\pi],\\
(\pi-a)^{-1}\bg(\frac{\alpha+3}{3}\bg)^{\frac{1}{\alpha}}(y-a),&y\in[a,\pi],\\
(\pi-a)^{-1}\bg(\frac{\alpha+3}{3}\bg)^{\frac{1}{\alpha}}(2\pi-a-y),&y\in[\pi,2\pi-a].
\end{array}
\right.
\end{align*}
One verifies by direct computation that $\rho\in H_y^1$ and
\begin{align*}
\|\rho\|_{L^2(\T)}^2<2\pi\quad\text{and}\quad\|\rho\|_{L^2(\T)}^2=\|\rho\|_{L^{\alpha+2}(\T)}^{\alpha+2}.
\end{align*}
Moreover, let $P$ be an optimizer of $\wm_1$. Particularly, by Pohozaev's identity and definition we have
\begin{gather*}
\|\nabla_x P\|_{L_x^2}^2=\frac{\alpha d}{2(\alpha+2)}\|P\|_{L_x^{\alpha+2}}^{\alpha+2},\\
\wm_1=\frac{1}{2}\wmM(P)+\frac{1}{2}\|\nabla_x P\|_{L_x^2}^2-\frac{1}{\alpha+2}\|P\|_{L_x^{\alpha+2}}^{\alpha+2}.
\end{gather*}
Now define $\psi(x,y):=\rho(y_1)\cdots\rho(y_m)P(x)$. Then
\begin{align*}
\mK(\psi)&=\|\nabla_x\psi\|_2^2-\frac{\alpha d}{2(\alpha+2)}\|\psi\|_{\alpha+2}^{\alpha+2}\nonumber\\
&=\|\rho\|_{L^2(\T)}^{2m}\|\nabla_x P\|_{L_x^2}^2
-\frac{\alpha d}{2(\alpha+2)}\|\rho\|_{L^{\alpha+2}(\T)}^{\alpha+2}\|P\|_{L_x^{\alpha+2}}^{(\alpha+2)m}\nonumber\\
&=\|\rho\|_{L^2(\T)}^{2m}\bg(\|\nabla_x P\|_{L_x^2}^2-\frac{\alpha d}{2(\alpha+2)}\|P\|_{L_x^{\alpha+2}}^{\alpha+2}\bg)=0
\end{align*}
and
\begin{align*}
\mN_{1,\infty}(\psi)
&=\frac{1}{2}\mM(\psi)+\frac{1}{2}\|\nabla_x \psi\|_{2}^2-\frac{1}{\alpha+2}\|\psi\|_{{\alpha+2}}^{\alpha+2}\nonumber\\
&=\frac{1}{2}\|\rho\|_{L^2(\T)}^{2m}\wmM(P)+\frac{1}{2}\|\rho\|_{L^2(\T)}^{2m}\|\nabla_x P\|_{L_x^2}^2
-\frac{1}{\alpha+2}\|\rho\|_{L^{\alpha+2}(\T)}^{(\alpha+2)m}\|P\|_{L_x^{\alpha+2}}^{\alpha+2}\nonumber\\
&=\|\rho\|_{L^2(\T)}^{2m}\bg(\frac{1}{2}\wmM(P)+\frac{1}{2}\|\nabla_x P\|_{L_x^2}^2-\frac{1}{\alpha+2}\|P\|_{L_x^{\alpha+2}}^{\alpha+2}\bg)<(2\pi)^m\wm_1.
\end{align*}
Consequently,
\begin{align*}
\lim_{\omega\to \infty}\gamma_{1,\omega}\leq \lim_{\omega\to \infty}\mN_{1,\omega}(\psi)=\mN_{1,\infty}(\psi)<(2\pi)^m\wm_1.
\end{align*}
That a minimizer of $\gamma_{1,\omega}$ has non-trivial $y$-dependence for $\omega>\omega_*$ follows already from the fact that $\gamma_{1,\omega}<(2\pi)^m \wm_1$ and the mapping $\omega\mapsto \gamma_{1,\omega}$ is monotone decreasing. We borrow an idea from \cite{GrossPitaevskiR1T1} to show that any minimizer of $\gamma_{1,\omega}$ for $\omega<\omega_*$ must be $y$-independent. Assume the contrary that there exists an optimizer $u_{1,\omega}$ of $\gamma_{1,\omega}$ satisfying $\|\nabla_y u_{1,\omega}\|_2^2\neq 0$. Let $\mu\in(\omega,\omega_*)$. Then
\begin{align*}
(2\pi)^m \wm_1=m_{1,\mu}\leq \mN_{1,\mu}(u_{1,\omega})=\mN_{1,\omega}(u_{1,\omega})+\frac{\mu-\omega}{2}\|\nabla_y u_{1,\omega}\|_2^2<\mN_{1,\omega}(u_{1,\omega})=\gamma_{1,\omega}=(2\pi)^m \wm_1,
\end{align*}
a contradiction. This completes the desired proof.
\end{proof}

We now prove Theorem \ref{thm threshold frequency} by using Lemma \ref{lemma auxiliary} and a simple rescaling argument.

\begin{proof}[Proof of Theorem \ref{thm threshold frequency}]
For $\kappa>0$, define the scaling operator $T_\kappa$ by
\begin{align}\label{def of t ld}
T_\kappa u(x,y):=\kappa^{\frac{2}{\alpha}}u(\kappa x,y).
\end{align}
Then
\begin{align*}
\|T_\kappa(\nabla_x u)\|_2^2&=\kappa^{2+\frac4\alpha-d}\|\nabla_x u\|_2^2,\\
\|T_\kappa u\|_{\alpha+2}^{\alpha+2}&=\kappa^{2+\frac4\alpha-d}\|u\|_{\alpha+2}^{\alpha+2},\\
\mK(T_\kappa u)&=\kappa^{2+\frac4\alpha-d}\mK(u),\\
\|T_\kappa (\nabla_y u)\|_2^2&=\kappa^{\frac4\alpha-d}\|\nabla_y u\|_2^2,\\
\|T_\kappa u\|_2^2&=\kappa^{\frac4\alpha-d}\|u\|_2^2.
\end{align*}
Let $\kappa=\sqrt{\omega}$. Direct computation shows that $\mN_\omega(T_\kappa u)=\omega^{1+\frac2\alpha-\frac{d}{2}}\gamma_{1,\omega^2}(u)$. This particularly implies $\gamma_{\omega}=\omega^{1+\frac2\alpha-\frac{d}{2}}\gamma_{1,\omega^2}$. By same rescaling arguments we also infer that $\omega^{1+\frac2\alpha-\frac{d}{2}}\wm_{1}=\wm_{\omega}$. Notice also that the mapping $\omega\mapsto \omega^{2}$ is strictly monotone increasing on $(0,\infty)$. Thus by Lemma \ref{lemma auxiliary} there exists some $\omega_*\in(0,\infty)$ such that
\begin{itemize}
\item For all $\omega \in(\omega_*,\infty)$ we have
$$\gamma_{\omega}=\omega^{1+\frac2\alpha-\frac{d}{2}}\gamma_{1,\omega^2}
<\omega^{1+\frac2\alpha-\frac{d}{2}}(2\pi)^m \wm_{1}=(2\pi)^m \wm_{\omega}.$$

\item For all $\omega \in(0,\omega_*)$ we have
$$\gamma_{\omega}=\omega^{1+\frac2\alpha-\frac{d}{2}}\gamma_{1,\omega^2}
=\omega^{1+\frac2\alpha-\frac{d}{2}}(2\pi)^m \wm_{1}=(2\pi)^m \wm_{\omega}.$$
\end{itemize}
The statements in Theorem \ref{thm threshold frequency}, up to the endpoint $\omega_*$, thus follow.

It remains to consider the case $\omega=\omega_*$, namely to prove $\gamma_{\omega_*}=(2\pi)^m\widehat \gamma_{\omega_*}$. We notice that this will immediately follow as long as the continuity of the mappings $\omega\mapsto \gamma_\omega$ and $\omega\mapsto \widehat\gamma_\omega$ on $(0,\infty)$ is proved. We shall only give the proof for the continuity statement of $\omega\mapsto \gamma_\omega$, the one of $\omega\mapsto \widehat\gamma_\omega$ is being similarly deduced. Let $\omega>0$ and let $(\omega_n)_n\subset(0,\infty)$ with $\omega_n\to\omega$. Let $u_\omega$ be an optimizer of $\gamma_\omega$. By definition and $\omega_n\to\omega$ we already have
\begin{align}\label{limsup cont}
\limsup_{n\to\infty}\gamma_{\omega_n}\leq \lim_{n\to\infty}S_{\omega_n}(u_\omega)=S_{\omega}(u_\omega)=\gamma_\omega.
\end{align}
Next, let $u_n$ be an optimizer of $\gamma_{\omega_n}$ that additionally solves \eqref{nls2} with the same given $\omega_n$. From \eqref{limsup cont} and $\omega_n\to\omega$ we know that $(\gamma_{\omega_n})_n$ and $(\omega_n)_n$ are bounded sequences and $\lim_{n\to\infty}\omega_n=\omega>0$. Then arguing as in \eqref{uniform bddnes} we infer that $(u_n)_n$ is a bounded sequence in $H_{x,y}^1$, whose weak limit is denoted by $u$. Moreover, using \eqref{key gn inq2} we may also assume that $u\neq 0$. Since $u_n$ solve \eqref{nls2} with frequency $\omega_n$ and $\omega_n\to \omega$, we infer that $u$ solves \eqref{nls2} with frequency $\omega$. Consequently, by Lemma \ref{lemma of poho} we know that $\mK(u)=0$. Finally, using weakly lower semicontinuity of norms and Lemma \ref{identification m gamma} we infer that
\begin{align*}
\gamma_{\omega}\leq \mN_\omega(u)=\mI_\omega(u)\leq\liminf_{n\to\infty}\mI_\omega(u_n)=\liminf_{n\to\infty}\mI_{\omega_n}(u_n)
=\liminf_{n\to\infty}\gamma_{\omega_n}.
\end{align*}
This completes the desired proof.
\end{proof}

\section{Identification of $\gamma_\omega$ and $\beta_\omega$: Proof of Proposition \ref{prop iden of gamma and beta}}\label{sec 3+}

\begin{proof}[Proof of Proposition \ref{prop iden of gamma and beta}]
Let first $u_\omega$ be an optimizer of $\gamma_\omega$, whose existence is guaranteed by Theorem \ref{thm existence ground states 1}. By Lemma \ref{vc implies sc} we know that $u_\omega$ is a solution of \eqref{nls2} with the given $\omega$, which also implies that $N_\omega(u_\omega)=0$. This in turn implies $\gamma_\omega=S_\omega(u_\omega)\geq \beta_\omega$.

On the other hand, we notice that the problem $\beta_\omega$ in the fractional setting has been considered in \cite{esfahani2023focusing}. Using the same arguments in \cite{esfahani2023focusing} with slight modification, one easily verifies that for any $\omega>0$ the variational problem $\beta_\omega$ (in our setting) also possesses an optimizer $v_\omega$ which solves \eqref{nls2} with the given $\omega$. By Lemma \ref{lemma of poho} we know that $K(v_\omega)=0$. Repeating the previous arguments yields $\beta_\omega\geq\gamma_\omega$, which in turn completes the proof.
\end{proof}

\section{The Legendre-Fenchel identity: Proof of Theorem \ref{thm legendre}}\label{sec 3}
In this subsection we give the proof of Theorem \ref{thm legendre}. The core of the proof is the following crucial lemma concerning the properties of the mapping $\omega\mapsto \gamma_\omega$.

\begin{lemma}\label{lemma 4.1}
The following statements hold:
\begin{itemize}
\item[(i)]There exists some $C>0$ such that for any $\omega>0$ we have $\gamma_\omega\leq C\omega^{1+\frac{2}{\alpha}-\frac{d+m}{2}}$.
\item[(ii)]For $\omega\in(0,\infty)$, define
\begin{align*}
\Gamma_{\mathrm{sup},\omega}:=\sup\{M(u):\,u\in\mathrm{argmin}\,\gamma_\omega\},\quad
\Gamma_{\mathrm{inf},\omega}:=\inf\{M(u):\,u\in\mathrm{argmin}\,\gamma_\omega\}.
\end{align*}
Then
\begin{align}\label{4.0}
\left\{
\begin{array}{cl}
\lim_{\omega\to 0}\Gamma_{\mathrm{inf},\omega}=\infty,&\text{if $\alpha\in(\tas,\tbs)$},\\
\\
\lim_{\omega\to 0}\Gamma_{\mathrm{inf},\omega}=2\pi \widehat M(Q),&\text{if $\alpha=\tas$ and $m=1$}
\end{array}
\right.
\end{align}
and $\lim_{\omega\to\infty}\Gamma_{\mathrm{sup},\omega}=0$.
\end{itemize}
\end{lemma}

\begin{proof}
We begin with the proof of (i). Assume that $0<\omega< \omega_*$, where $\omega_*$ is the number given by Theorem \ref{thm threshold frequency}. Then using Theorem \ref{thm threshold frequency} and simple scaling arguments we know that
\begin{align}\label{4.1}
\gamma_\omega=(2\pi)^m\widehat \gamma_\omega=(2\pi)^m\omega^{1+\frac2\alpha-\frac{d}{2}}\wm_{1}
\end{align}
holds for $\omega\in(0,\omega_*)$. Hence $\gamma_\omega\leq C_1\omega^{1+\frac{2}{\alpha}-\frac{d+m}{2}}$ for $\omega\in(0,\omega_*)$ with
$C_1:=\widehat\gamma_1\omega_*^{\frac{m}{2}}$.

It remains to consider the case $\omega\geq\omega_*$. As the mapping $\omega\mapsto \gamma_\omega$ is continuous on $(0,\infty)$ (see the proof of Theorem \ref{thm existence ground states 1}), by standard continuity arguments it suffices to consider the case $\omega\gg 1$. We first claim that there exists some $\phi\in C_0^\infty(\R^{d+m})$ supported in $B_1(0):=\{z\in\R^{d+m}:|z|\leq 1\}$ such that $K(\phi)=0$. Here we have identified $\phi$ to a function in $H_{x,y}^1$ by periodically extending $\phi$ along $\T^m$. To do so, we may simply pick an arbitrary function $\phi\in C_0^\infty(\R^{d+m})$ supported in $B_1(0)$. If $K(\phi)=0$ then we are done. Otherwise we may find some $t\in(0,\infty)$ such that $K(t\phi)=0$, and we simply replace $\phi$ to $t\phi$.

Next, define the scaling operator $\mathcal{L}_\mu$ by
\[\mathcal{L}_\mu u(x,y):=\mu^{\frac{2}{\alpha}}u(\mu x,\mu y).\]
Thus if $u$ is supported in $B_1(0)$, then $\mathcal{L}_\mu u$ still defines a function in $H_{x,y}^1$ for any $\mu\geq 1$. Notice particularly that $K(\mathcal{L}_\mu \phi)=0$ for any $\mu>0$. Hence by definition
\[\gamma_\omega\leq S_\omega(\mathcal{L}_\mu\phi)=\frac{1}{2}\omega \mu^{\frac{4}{\alpha}-(d+m)}\|\phi\|_{L^2(\R^{d+m})}^2
+\mu^{\frac{4}{\alpha}-(d+m-2)}\tilde{E}(\phi)=:f_\omega(\mu),\]
where
\[\tilde E(\phi)=\frac{1}{2}\|\nabla_{x,y}\phi\|_{L^2(\R^{d+m})}^2
-\frac{1}{\alpha+2}\|\phi\|_{L^{\alpha+2}(\R^{d+m})}^{\alpha+2}.\]
Arguing as in Corollary \ref{cor lower bound} (but rather using the classical Gagliardo-Nirenberg inequality on $\R^{d+m}$) we know that $K(\phi)=0$ implies $\tilde E(\phi)\in(0,\infty)$. Notice that $\alpha\in[\tas,\tbs)$ implies $\frac{4}{\alpha}-(d+m)<0$ and $\frac{4}{\alpha}-(d+m-2)>0$. Thus minimizing the function $f_\omega(\mu)$ in $\mu$ yields
$$f_\omega(\mu_{\min})=\min_{\mu>0}f_\omega(\mu)\sim \omega^{1+\frac{2}{\alpha}-\frac{d+m}{2}}.$$
Note that to achieve the minimum we need to choose $\mu_{\min}\sim \sqrt{\omega}$. Since we are considering the regime $\omega\gg 1$, by this choice of $\mu$ the function $\mathcal{L}_{\mu_{\min}}\phi$ will still be a candidate in $H_{x,y}^1$. This in turn completes the proof of (i).

Now we prove (ii) and begin with proving $\lim_{\omega\to\infty}\Gamma_{\mathrm{sup},\omega}=0$. Let $u_\omega$ be an arbitrary optimizer of $\gamma_\omega$. Notice that
\[\frac{\omega}{2}M(\omega)\leq
\frac{\omega}{2}M(\omega)+\frac{1}{2}\|\nabla_y u_\omega\|_2^2+
\bg(\frac{1}{2}-\frac{2}{\alpha d}\bg)\|\nabla_x u_\omega\|_2^2
=S_\omega(u_\omega)-\frac{2}{\alpha d}K(u_\omega)=\gamma_\omega.\]
Thus using also (i)
\[M(u_\omega)\leq \frac{2\gamma_\omega}{\omega}\lesssim \omega^{\frac{2}{\alpha}-\frac{d+m}{2}}\to 0\]
as $\omega\to\infty$, since $\frac{2}{\alpha}-\frac{d+m}{2}<0$. As $u_\omega$ is chosen arbitrarily, this proves $\lim_{\omega\to\infty}\Gamma_{\mathrm{sup},\omega}=0$.

Finally we prove \eqref{4.0}. First consider the case $\alpha\in(\tas,\tbs)$. Notice that by Theorem \ref{thm threshold frequency} we have $\nabla_y u_\omega=0$ for all $\omega\ll 1$, which we assume in the rest of the proof. Applying the Pohozaev's identity on $\R^d$ we obtain
\begin{align*}
\omega\|u_\omega\|_2^2+\|\nabla_x u_\omega\|_2^2-\|u\|_{\alpha+2}^{\alpha+2}&=0,\\
\|\nabla_x u_\omega\|_2^2-\frac{\alpha d}{2(\alpha+2)}\|u_\omega\|_{\alpha+2}^{\alpha+2}&=0,
\end{align*}
from which we deduce that
\[\omega M(u_\omega)=\bg(\frac{2(\alpha+2)}{\alpha d}-1\bg)\|\nabla_x u_\omega\|_2^2.\]
Inserting these identities into $\gamma_\omega=S_\omega(u_\omega)$ and using \eqref{4.1} yields
\begin{align}\label{new poho}
\omega^{1+\frac2\alpha-\frac{d}{2}}\wm_{1}=\gamma_\omega=\frac{\alpha\omega}{4-\alpha(d-2)}M(u_\omega).
\end{align}
Thus $M(u_\omega)\sim \omega^{\frac{2}{\alpha}-\frac{d}{2}}$. The desired claim follows by noticing that $\alpha>\frac{4}{d}$ implies $\frac{2}{\alpha}-\frac{d}{2}<0$.

We now consider the case $\alpha=\tas$ and $m=1$. By Theorem \ref{thm threshold frequency} we know that for all $0<\omega\ll 1$ the ground state solutions $u_\omega$ will reduce to the ones on $\R^d$. However, by the classical theories for ground state solutions of the mass-critical NLS on $\R^d$ (see e.g. \cite{weinstein,Kwong_uniqueness}) we know that all such solutions have the fixed mass $\widehat M(Q)$ in the space $\R^d$. This in turn completes the desired proof.
\end{proof}

We are now ready to prove Theorem \ref{thm legendre}.

\begin{proof}[Proof of Theorem \ref{thm legendre}]
That in the case $m=0$ we have $A_c\neq\varnothing$ for any $c\in(0,\infty)$ follows immediately from \eqref{new poho}. In the case $m=1$, that $A_c\neq \varnothing$ for some $c=c_n\in(0,\frac1n)$ follows from $\lim_{\omega\to\infty}\Gamma_{\mathrm{sup},\omega}=0$ deduced in Lemma \ref{lemma 4.1}. This completes the first part of the proof of \ref{thm legendre}.

We next prove \eqref{lengendre identity} for mass numbers $c$ satisfying $A_c\neq\varnothing$. By definition and $A_c\neq\varnothing$ we already know that $m_c\leq \gamma_\omega-\frac{1}{2}c\omega$. Now let $(u_n)_n\subset V(c)$ be a minimizing sequence of $m_c$. For $\omega\in A_c$ we deduce by definition that
\begin{align}\label{abc}
E(u_n)=S_\omega(u_n)-\frac{1}{2}\omega c\geq \gamma_{\omega}-\frac{1}{2}c\omega.
\end{align}
Sending $n\to\infty$ then yields the desired claim.
\end{proof}

\section{Existence of normalized mass-critical ground states: Proof of Theorem \ref{thm threshold mass 2}}\label{sec 4}
As an important application of Theorem \ref{thm legendre}, we establish the existence of normalized ground states for the mass-critical variational problem $m_c$.

\begin{proof}[Proof of Theorem \ref{thm threshold mass 2}]
By Theorem \ref{thm legendre} we already know that for $\omega\in A_c$, an optimizer $u_\omega$ of $\gamma_\omega$ will also be an optimizer of $m_c$. (i) and (ii) then follow immediately. Notice now that if $\nabla_y u=0$, then we necessarily have $M(u)=(2\pi)^m \widehat M(Q)>c$, a contradiction. This completes the proof of (iii) and in turn also the proof of Theorem \ref{thm threshold mass 2}.
\end{proof}

\section{Properties of the mapping $c\mapsto m_c$: Proof of Proposition \ref{prop monotone}}\label{sec 5}
The proof of Proposition \ref{prop monotone} is composed of the following two lemmas.

\begin{lemma}\label{lem 6.1}
Let $\alpha=\tas$ and $m=1$. Then the set
\[V(c)=\{u\in S(c):K(u)=0\}\]
is not empty for any $(0,2\pi \widehat M(Q))$. Moreover, the mapping $c\mapsto m_c$ is monotone decreasing on $(0,2\pi \widehat M(Q))$.
\end{lemma}

\begin{remark}
We point out that Lemma \ref{lem 6.1} will indeed hold for other combinations of $\alpha,d,m$ satisfying \eqref{add} (where the range for $c$ is replaced by $(0,\infty)$). We will focus here on the mass-critical case where the standard scaling arguments are not available. For a proof of the other combinations where the model is mass-supercritical, we refer e.g. to \cite[Lem. 2.4]{Luo_inter}.
\end{remark}

\begin{proof}
We first prove that the set $V(c)$ is not empty. By Lemma \ref{lemma 4.1} we can find some $\omega>0$ possibly large such that the optimizer $u_\omega$ of $\gamma_\omega$ satisfies $M(u_\omega)<c$. Now take some $t>1$ possibly close to $1$ such that
\[M(t u_\omega)<c\quad\text{and}\quad K(tu_\omega)<0,\]
where the latter inequality is valid for all $t>1$ since $K(u_\omega)=0$. Next, let $\eta\in C_c^\infty(\R^d;[0,1])$ be a radially decreasing cut-off function such that $\eta\equiv 1$ in $B_1(0)$ and $\eta\equiv 0$ in $B_2^c(0)$. We then choose $R\gg 1$ such that
\[M(\eta(\cdot/R) t u_\omega)<c\quad\text{and}\quad K(\eta(\cdot/R) t u_\omega)<0.\]
Finally, let $\phi\in C_c^\infty(\R^d)$ such that $M(\phi)=c-M(\eta(\cdot/R) t u_\omega)$. Using the Gagliardo-Nirenberg inequality on $\R^d$ we know that $K(\phi)>0$. By noticing $K(s^{\frac{d}{2}} \phi(s\cdot))=s^2 K(\phi)$ we may tune the value $s\in(0,\infty)$ such that
\[K(s^{\frac{d}{2}} \phi(s\cdot))+K(\eta(\cdot/R) t u_\omega)=0.\]
Fixing such $s$ and set $u:=s^{\frac{d}{2}} \phi(s\cdot)+\eta_R t u_\omega$. Since the two summands are compactly supported in the $x$-space, by further translating them along $x$-direction we conclude that $M(u)=c$ and $K(u)=0$ (where we also used $M(s^{\frac{d}{2}} \phi(s\cdot))=M(\phi)$) and in turn that $u\in V(c)$, as desired.

We now prove the monotonicity of the mapping $c\mapsto m_c$. For $2\pi \widehat M(Q)>a>b>0$ we need to show $m_{a}<m_b$. Let $\vare>0$ be some small number to be chosen later. Since $V(b)$ is not empty, we can find some $u_b\in V(b)$ such that $E(u_b)+\frac{\vare}{2}\geq E(u_b)$. Next, let $\phi\in C_c^\infty(\R^d)$ satisfy $\widehat M(\phi)=(2\pi)^{-1}(a-b)<\widehat M(Q)$. Since $\widehat M(\phi)<\widehat M(Q)$, using again the Gagliardo-Nirenberg inequality on $\R^d$ we infer that $K(\phi)>0$. For $t>1$ we know that $K(t u_b)<0$. Choose now $t=t(\vare)>1$ such that there exists $\beta=\beta(\vare)>0$ such that
\begin{align}\label{6.1}
m_b+\vare\geq E(tu_b),\quad b+\vare\geq M(tu_b)>b
,\quad  K(tu_b)\in(-2\beta,-\beta).
\end{align}
We also choose $R=R(\vare)\gg 1$ such that
\begin{gather}\label{6.2}
\begin{array}{cc}
E(tu_b)+\vare\geq E(\eta(\cdot/R)tu_b)
,\quad
M(tu_b)\geq M(\eta(\cdot/R)tu_b)>b,\quad
K(\eta(\cdot/R)tu_b)\in(-3\beta,-\frac{\beta}{2}).
\end{array}
\end{gather}
Defining $u_1:=\eta(\cdot/R)tu_b$, \eqref{6.1} and \eqref{6.2} imply that
\begin{align}\label{6.3}
m_1+2\vare\geq E(u_1),\quad b+\vare\geq M(u_1), \quad K(u_1)\in(-3\beta,-\frac{\beta}{2}).
\end{align}
Next, choose $s=s(\vare)\in(0,1)$ such that $M(s\phi)+M(u_1)=a$. Define
\[\phi^q(x):=q^{\frac{d}{2}}s\phi(qx).\]
By direct computation we see that $K(\phi^q)=q^2 K(s\phi)$. Since $s<1$, we also know that $K(s\phi)>0$. Combining with $K(u_1)<0$, we may tune the value $q$ in a way such that $K(\phi^q)+ K(u_1)=0$. Define now $u_2:=\phi^q$. Since $u_1$ and $u_2$ have compact supports, up to an $x$-spatial translation we may w.l.o.g. assume that $u_1$ and $u_2$ have disjoint supports. Finally, set $u_3:=u_1+u_2$. Combining with $M(\phi^q)=M(s\phi)$ and the fact that $\phi^q$ is $y$-independent we conclude that
\begin{align*}
M(u_3)=a,\quad K(u_3)=0
\end{align*}
and
\begin{align*}
E(u_3)=E(u_3)-\frac{1}{2}K(u_3)=\frac{1}{2}\|\nabla_y u_3\|_2^2=\|\nabla_y u_1\|_2^2=E(u_1)\leq m_b+2\vare.
\end{align*}
By definition of $m_a$ we obtain
\[m_a\leq E(u_3)\leq m_b+2\vare.\]
The desired claim follows by choosing $\vare$ arbitrarily small.
\end{proof}

\begin{lemma}\label{lem 6.2}
The mapping $c\mapsto m_c$ is lower semicontinuous on $(0,2\pi \widehat M(Q))$.
\end{lemma}

\begin{proof}
Since $c\mapsto m_c$ is monotone decreasing, it suffices to prove the right continuity of $c\mapsto m_c$. Let thus $c_n\searrow c$. Let also $u_n\in V(c_n)$ such that $E(u_n)\leq m_{c_n}+n^{-1}$. Since $K(u_n)=0$, by rescaling we may assume that $\|\nabla_x u_n\|_2^2=\frac12\|u_n\|_{\tas+2}^{\tas+2}=1$. Hence $(u_n)_n$ is a bounded sequence in $H_{x,y}^1$ and using Lemma \ref{lemma non vanishing limit} we know that (up to translations) $u_n$ has a weak $u\neq 0$ in $H_{x,y}^1$. Particularly,
$$\bar c:=M(u)\leq \liminf_{n\to\infty} M(u_n)=\liminf_{n\to\infty} c_n=c.$$
Arguing as in the proof of Theorem \ref{thm existence ground states 1} we know that $u$ is an optimizer of $m_{\bar c}$. Combining with the monotonicity of $c\mapsto m_c$ and the weak lower semicontinuity of norms we conclude that
\[m_c\leq m_{\bar c}=\frac12\|\nabla_y u\|_2^2\leq \frac12\liminf_{n\to\infty}
\|\nabla_y u_n\|_2^2=\liminf_{n\to\infty} m_{c_n},\]
as desired.
\end{proof}

\begin{proof}[Proof of Proposition \ref{prop monotone}]
This follows immediately from Lemma \ref{lem 6.1} and \ref{lem 6.2}.
\end{proof}

\section{Scattering versus blow-up: Proof of Theorem \ref{thm large data} and \ref{thm blow up}} \label{sec 6}
In this section we prove Theorem \ref{thm large data} and \ref{thm blow up}. Throughout this section, we fix $\alpha=\tas$ and denote by $(\mathbf{a},\mathbf{r})$ an admissible pair. For $d=2$, we fix $(\mathbf{a},\mathbf{r})$ to be $(4,4)$.

\subsection{Some preliminaries}
To begin with, we firstly collect some useful lemmas.

\begin{lemma}[Small data well-posedness, \cite{TzvetkovVisciglia2016}]\label{lemma cnls well posedness} Let $I$ be an open interval containing $0$. Define
\begin{align*}
X(I)&:=(L_t^\infty L_x^2 {H}_y^1(I)\cap L_t^{2^+} L_x^{(\tbs+2)^-} {H}_y^1(I))\cap (L_t^\infty \dot{H}_x^1 L_y^2(I)
\cap L_t^{2^+} \dot{W}_x^{1,(\tbs+2)^-} L_y^2(I))\\
&=:S_0 {H}_y^1(I)\cap S_1 L_y^2(I).
\end{align*}
Let also $s\in(\frac12,1]$. Assume that
\begin{align*}
\|u_0\|_{H_{x,y}^1}\leq A
\end{align*}
for some $A>0$. Then there exists $\delta=\delta(A)$ such that if
\begin{align*}
\|e^{it\Delta}u_0\|_{\diag  H_y^{s}(I) }\leq \delta,
\end{align*}
then there exists a unique solution $u\in X(I)$ of \eqref{nls} with $u(0)=u_0$ such that
\begin{align*}
\|u\|_{X(I)}\lesssim A\quad\text{and}\quad
\|u\|_{\diag  H_y^{s}(I)}\leq 2\|e^{it\Delta}u_0\|_{\diag  H_y^{s}(I)}.
\end{align*}
\end{lemma}

\begin{lemma}[Scattering criterion, \cite{Luo_Waveguide_MassCritical}]\label{scattering crit}
If $u$ is a global solution of \eqref{nls} and there exists some $s\in(\frac12,1]$ such that
\begin{align*}
\|u\|_{\diag  H_y^{s}(\R)}+\|u\|_{L_t^\infty H_{x,y}^1(\R)}<\infty,
\end{align*}
then $u$ is scattering in $H_{x,y}^1$.
\end{lemma}

\begin{lemma}[Long time stability, \cite{CubicR2T1Scattering}]\label{lem stability cnls}
Let $d=2$. Let also $u$ be a solution of \eqref{nls} on the time interval $I\ni 0$ and let $z$ be a solution of
\begin{align*}
i\pt_t z+\Delta_{x,y} z=-|z|^2 z+e
\end{align*}
on $I$. Let also $s\in(\frac12,1]$ be given. Assume that there exist $M,L,M'>0$ such that
\begin{align*}
\|u\|_{L_t^\infty L_x^2 H_y^{s}(I)}&\leq M,\\
\|z\|_{L_{t,x}^4 H_y^{s} (I)}&\leq L,\\
\|z(0)-u(0)\|_{L_x^2 H_y^{s}}&\leq M'.
\end{align*}
Assume also the smallness conditions
\begin{alignat*}{2}
\|e^{it\Delta}(z(0)-u(0) )&\|_{L_{t,x}^4 H_y^{s}(I)}&&\leq \vare,\\
\|e&\|_{L_{t,x}^{\frac{4}{3}}H_y^{s}(I)}&&\leq \vare
\end{alignat*}
for some $0<\vare\leq \vare_1$, where $\vare_0=\vare_0(M,M',L)>0$ is a small constant. Then
\begin{align*}
\|z-u\|_{L_{t,x}^4 H_y^{s}(I)}&\leq C(M,M',L)\vare,\\
\|z-u\|_{S_0H_y^{s}(I)}&\leq C(M,M',L)M',\\
\|z\|_{S_0H_y^{s}(I)}&\leq C(M,M',L).
\end{align*}
\end{lemma}

\begin{lemma}[Linear profile decomposition for bounded $H_{x,y}^1$-sequence]\label{linear profile}
Let $(\psi_n)_n$ be a bounded sequence in $H_{x,y}^1$. Then up to a subsequence, there exist nonzero linear profiles
$(\tdu^j)_j\subset L_x^2 H_y^1$, remainders $(w_n^k)_{k,n}\subset H_{x,y}^1$, parameters $(t^j_n,x^j_n,\xi^j_n,\ld^j_n)_{j,n}\subset\R\times\R^d\times\R^d\times(0,\infty)$ and $K^*\in\N\cup\{\infty\}$, such that
\begin{itemize}
\item[(i)] For any finite $1\leq j\leq K^*$ the parameters satisfy
\begin{align*}
1&\gtrsim_j\lim_{n\to\infty}|\xi_n^j|,\nonumber\\
\lim_{n\to\infty}t^j_n&=:t_\infty^j\in\{0,\pm\infty\},\nonumber\\
\lim_{n\to\infty}\ld^j_n&=:\ld_\infty^j\in\{1,\infty\},\nonumber\\
t_n^j&\equiv 0\quad\text{if $t_\infty^j=0$},\nonumber\\
\ld_n^j&\equiv 1\quad\text{if $\ld_\infty^j=1$},\nonumber\\
\xi_n^j&\equiv 0\quad\text{if $\ld_\infty^j=1$}.
\end{align*}

\item[(ii)]For any finite $1\leq k\leq K^*$ we have the decomposition
\begin{align*}
\psi_n=\sum_{j=1}^k T^j_nP_n^j \tdu^j+w_n^k.
\end{align*}
Here, the operators $T_n^j$ and $P_n^j$ are defined by
\begin{align*}
T^j_n u(x):=
\left\{
             \begin{array}{ll}
             [e^{it^j_n\Delta_x}u](x-x^j_n,y),&\text{if $\ld^j_\infty=1$},\\
             \\
             g_{\xi^j_n,x^j_n,\ld^j_n}[e^{it^j_n\Delta_x}u](x,y),&\text{if $\ld^j_\infty=\infty$}
             \end{array}
\right.
\end{align*}
and
\begin{align*}
P^j_n u:=
\left\{
             \begin{array}{ll}
             u,&\text{if $\ld^j_\infty=1$},\\
             \\
             P_{\leq(\ld_n^j)^\theta}u,&\text{if $\ld^j_\infty=\infty$}
             \end{array}
\right.
\end{align*}
for some $\theta\in(0,1)$. Moreover,
\begin{align*}
\tdu^j\in
\left\{
             \begin{array}{ll}
             H_{x,y}^1,&\text{if $\ld^j_\infty=1$},\\
             \\
             L_x^2 H_y^1,&\text{if $\ld^j_\infty=\infty$}.
             \end{array}
\right.
\end{align*}

\item[(iii)] For $s\in(\frac12,1)$ the remainders $(w_n^k)_{k,n}$ satisfy
\begin{align*}
\lim_{k\to K^*}\lim_{n\to\infty}\|e^{it\Delta_{x,y}}w_n^k\|_{\diag H_y^s(\R)}=0.
\end{align*}

\item[(iv)] The parameters are orthogonal in the sense that
\begin{align*}
 \frac{\ld_n^k}{\ld_n^j}+ \frac{\ld_n^j}{\ld_n^k}+\ld_n^k|\xi_n^j-\xi_n^k|+\bg|t_k\bg(\frac{\ld_n^k}{\ld_n^j}\bg)^2-t_n^j\bg|
+\bg|\frac{x_n^j-x_n^k-2t_n^k(\ld_n^k)^2(\xi_n^j-\xi_n^k)}{\ld_n^k}\bg|\to\infty
\end{align*}
for any $j\neq k$.

\item[(v)] For any finite $1\leq k\leq K^*$ and $D\in\{1,\nabla_x,\nabla_y\}$ we have the energy decompositions
\begin{align*}
\|D\psi_n\|_{2}^2&=\sum_{j=1}^k\|D(T_n^jP_n^j\tdu^j)\|_{2}^2+\| Dw_n^k\|_{2}^2+o_n(1),\\
\|\psi_n\|_{\tas+2}^{\tas+2}&=\sum_{j=1}^k\|T_n^jP_n^j\tdu^j\|_{\tas+2}^{\tas+2}+\|w_n^k\|_{\tas+2}^{\tas+2}+o_n(1)
\end{align*}
\end{itemize}
\end{lemma}

\begin{lemma}[Large scale approximation, \cite{CubicR2T1Scattering,Luo_Waveguide_MassCritical,similar_cubic}]\label{cnls lem large scale proxy}
Fix $d=2$. Let $(\ld_n)_n\subset(0,\infty)$ be given such that $\ld_n\to \infty$, $(t_n)_n\subset\R$ be given such that either $t_n\equiv 0$ or $t_n\to\pm\infty$ and $(x_n,\xi_n)_n\subset\R^2\times\R^2$ be given such that $(\xi_n)_n$ is bounded. Let $\phi\in L_{x}^2 H_y^1$ and define
$$\phi_n:=g_{\xi_n,x_n,\ld_n}e^{it_n\Delta}P_{\leq \ld_n^\theta}\tdu$$
for some $\theta\in(0,1)$. Assume also that $\mM(\phi)<\pi \widehat M(Q)$. Then for all sufficiently large $n$ the solution $u_n$ of \eqref{nls} with $u_n(0)=\phi_n$ is global and scattering in time with
\begin{align}
\limsup_{n\to\infty}\|u_n\|_{L_{t,x}^4 H_y^1(\R)}&\leq C(\|\phi\|_{L_x^2 H_y^1}).\label{L2 proxy 1}
\end{align}
Furthermore, for every $\beta>0$ there exists $N_\beta\in\N$ and $\psi_\beta\in C_c^\infty(\R\times \R^2)\otimes C_{\mathrm{per}}^\infty(\T)$ such that
\begin{align}
\bg\|u_n-\ld_n^{-1}e^{-it|\xi_n|^2}e^{i\xi_n\cdot x}\psi_\beta\bg(\frac{t}{\ld_n^2}+t_n,\frac{x-x_n-2t\xi_n}{\ld_n},y\bg)\bg\|_{L_{t,x}^4 H_y^{1}(\R)}\leq \beta\label{L2 proxy 3}
\end{align}
for all $n\geq N_\beta$.
\end{lemma}

\subsection{Variational analysis and the MEI-functional}\label{subsec: r2t1 variational}
For $c\in(0,2\pi\widehat M(Q))$ and $\nu\in(0,m_c)$, define
\begin{align}\label{7.26}
\mathcal{A}_{c,\nu}:=\{u\in H_{x,y}^1:M(u)< c,\, E(u)<m_c-\nu,\,K(u)>0\}.
\end{align}
As will be seen later, we will always fix chosen $c$ and $\nu$ to manipulate our analysis. For this reason we simply drop the subscripts $c$ and $\nu$ from the set $\mathcal{A}_{c,\nu}$ in the rest of the paper.

We next establish the following coerciveness result for elements from the set $\mathcal{A}$.

\begin{lemma}\label{lemma coercivity}
Let $u\in \mathcal{A}$. Then there exists some $\delta\in(0,\frac12)$ such that
\[\left(\frac12-\delta\right)\|\nabla_{x,y} u\|_2^2\leq E(u)\leq \frac12\|\nabla_{x,y}u\|_2^2.\]
\end{lemma}

\begin{proof}
The upper bound is trivial. Next, notice that $K(u)>0$ implies that there exists some $\delta_u\in(0,1)$ such that \[\frac{1}{\tas+2}\|u\|_{\tas+2}^{\tas+2}= \frac{\delta_u}{2}\|\nabla_x u\|_2^2.\]
It suffices to show $\sup_{u\in\mathcal{A}}\delta_u<1$. Thus assume that there exists $(u_n)_n\subset \mathcal{A}$ and $(\delta_n)_n=(\delta_{u_n})_n$ such that $\delta_n\to 1$. We exploit the scaling operator $t\mapsto u^t(x,y):=t^{\frac{d}{2}}u(tx,y)$ to tune the sequence $(u_n)_n$ such that $\|\nabla_x u_n^{t_n}\|_2^2=1$. Since $\|\pt_y u_n\|_2^2$ and $M(u_n)$ are invariant w.r.t. the scaling operator, using $E(u_n)<m_c$ and $K(u_n)>0$ we know that
\[\|\pt_y u_n^{t_n}\|_2^2=\|\pt_y u_n\|_2^2<2m_c<\infty\]
and $M(u_n^{t_n})=M(u_n)<c$. Thus $(u_n^{t_n})_n$ is a bounded sequence in $H_{x,y}^1$. Moreover, using $\delta_n\to 1$ and $\frac{1}{\tas+2}\|u_n^{t_n}\|_{\tas+2}^{\tas+2}= \frac{\delta_n}{2}\|\nabla_x u_n^{t_n}\|_2^2=\frac{\delta_n}{2}$ we infer that $\liminf_{n\to\infty}\|u_n^{t_n}\|_{\tas+2}^{\tas+2}>0$. Thus $(u_n^{t_n})_n$ weakly converges to some $u\in H_{x,y}^1\setminus\{0\}$ in $H_{x,y}^1$ with $M(u)\leq c$. We next claim that $K(u)=0$.
\begin{itemize}
\item Suppose first that $K(u)<0$. Then there exists some $t\in(0,1)$ such that $K(tu)=0$. However, using the monotonicity of $c\mapsto m_c$ and the weakly lower semicontinuity of norms we obtain that
\[m_c\leq m_{M(tu)}\leq \frac{t^2}{2}\|\pt_y u\|_2^2<\frac12\|\pt_y u\|_2^2\leq \liminf_{n\to\infty}\bg(\frac12\|\pt_y u_n\|_2^2\bg)\leq m_c-\nu,\]
a contradiction.

\item Suppose now that $K(u)>0$. By noticing $K(u_n^{t_n})=o_n(1)$ we may argue as in the proof of Theorem \ref{thm existence ground states 1} to infer that $K(s_nu_n^{t_n}-u)<0$, $M(s_nu_n^{t_n}-u)<c$ and $\|\pt_y(s_nu_n^{t_n}-u)\|_2^2\leq 2m_c-\|\pt_y u\|_2^2+o_n(1)$ for all $n\gg 1$. But then we reach to the previous contradiction again.
\end{itemize} 
Thus $K(u)=0$. Consequently,
\[m_c\leq m_{M(u)}=\frac12 \|\pt_y u\|_2^2\leq\liminf_{n\to\infty}\bg(\frac12\|\pt_y u_n\|_2^2\bg)\leq m_c-\nu,\]
which is again a contradiction. The proof is therefore complete.
\end{proof}

The NLS-flow also leaves the elements in $\mA$ invariant, as shown in the following lemma.

\begin{lemma}\label{lem 7.8}
Let $u$ be a solution of \eqref{nls}. If $u(0)\in\mA$, then $u(t)\in\mA$ for all $t$ lying in the maximal lifespan of $u$.
\end{lemma}

\begin{proof}
Assume that this is not the case. By conservation of mass and energy this can only happen if there is some $t$ with $K(u(t))\leq 0$. By continuity we can then find some $t_0$ lying between $0$ and $t$ such that $K(u(t_0))=0$. By Proposition \ref{prop monotone} we also know that $E(u(t_0))<m_c= m_{M(u(t_0))}$, which however contradicts the minimality of $m_{M(u(t_0))}$.
\end{proof}

\begin{remark}
In view of Lemma \ref{lem 7.8}, we simply write $u\in\mA$ for a solution $u$ of \eqref{nls} if there is some $t$ in the lifespan of $u$ with $u(t)\in\mA$.
\end{remark}

We now construct the mass-energy-indicator (MEI) functional $\mathcal{D}$. Let $\Omega:=[0,c)\times [0,m_c-\nu)$. We then define the MEI-functional $\mathcal{D}:\Omega\to(0,\infty)$ by
\[\mathcal{D}(a,b)=\frac{a}{c-a}+\frac{b}{m_c-\nu-b}.\]
For $u\in H_{x,y}^1$, we define $\mathcal{D}(u)$ as $\mathcal{D}(u):=\mathcal{D}(M(u),E(u))$. Notice that by the conservation laws of NLS, if $u$ is a solution of \eqref{nls}, then $\mathcal{D}(u)$ will also be a conserved quantity. In this case, we may simply write $\mathcal{D}(u(t))\equiv \mathcal{D}(u)$ for any $t$ lying in the lifespan of $u$.

We next establish some properties of the MEI-functional $\mathcal{D}$.

\begin{lemma}\label{cnls killip visan curve}
Let $u,u_1,u_2$ be functions in $H_{x,y}^1$. The following statements hold true:
\begin{itemize}
\item[(i)] If $u\in\mA$ then $\mD(u)\in(0,\infty)$.

\item[(ii)] Let $u_1,u_2\in \mA$ satisfy $\mM(u_1)\leq \mM(u_2)$ and $\mH(u_1)\leq \mH(u_2)$, then $\mD(u_1)\leq \mD(u_2)$. If in addition either $\mM(u_1)<\mM(u_2)$ or $\mH(u_1)<\mH(u_2)$, then $\mD(u_1)<\mD(u_2)$.

\item[(iii)] For all $u\in \mA$ it holds $\|u\|^2_{H_{x,y}^1}\leq \bg(\frac12-\delta\bg)^{-1}(m_c+c-\nu)\mD(u)$.
\end{itemize}
\end{lemma}

\begin{proof}
(i) and (ii) follow immediately from the definition of the MEI-functional and the fact that $E(u)\in (0,m_c-\nu)$ for $u\in\mA$. For (iii), first estimate $M(u)$ by
\[\mD(u)=\frac{M(u)}{c-M(u)}+\frac{E(u)}{m_c-\nu-E(u)}\geq \frac{M(u)}{c-M(u)},\]
from which we infer that $M(u)\leq c(1+\mD(u))^{-1}\mD(u)\leq c\mD(u)$. Similarly we have $E(u)\leq (m_c-\nu)\mD(u)$. The desired claim follows by also taking Lemma \ref{lemma coercivity} into account.
\end{proof}

\subsection{Existence of a minimal blow-up solution}
Throughout this section the dimension number $d$ will be set to $d=2$.

For $c\in(0,\pi\widehat M(Q))$ and $\nu\in(0,m_c)$ let $\mA=\mA_{c,\nu}$ be the set defined through \eqref{7.26}. Thanks to the monotonicity and lower semicontinuity of the mapping $c\mapsto m_c$ deduced in Proposition \ref{prop monotone}, the scattering result given in Theorem \ref{thm large data} will follow as long as for each given $c\in(0,\pi\widehat M(Q))$ and $\nu\in(0,m_c)$ we can prove the same statement for all initial data lying in $\mA$. This will be the main task in the rest of the paper.

Define now
\begin{align*}
\tau(\mD_0):=\sup\bg\{\|u\|_{\diag H_y^{s}(I_{\max})}:
\text{ $u$ is solution of \eqref{nls}, $u\in\mA$, $\mD(u)\in (0,\mD_0)$}\bg\}
\end{align*}
and
\begin{align}\label{introductive hypothesis}
\mD^*&:=\sup\{\mD_0>0:\tau(\mD_0)<\infty\}.
\end{align}
By Lemma \ref{lemma cnls well posedness}, \ref{scattering crit}, \ref{lemma coercivity} and \ref{cnls killip visan curve} we know that $\mD^*>0$ and $\tau(\mD_0)<\infty$ for sufficiently small $\mD_0$. Consequently, we may simply assume $\mD^*<\infty$ and derive from this a contradiction, which leads to the desired claim.

Thus from now on let the assumption $\mD^*\in(0,\infty)$ be settled. By the inductive hypothesis we can find a sequence $(u_n)_n$ which are solutions of \eqref{nls} with $(u_n(0))_n\subset {\mA}$ and maximal lifespan $(I_{n})_n$ such that
\begin{gather}
\lim_{n\to\infty}\|u_n\|_{L_{t,x}^4 H_y^{s}((\inf I_n,0])}=\lim_{n\to\infty}\|u_n\|_{L_{t,x}^4 H_y^{s}([0, \sup I_n))}=\infty,\label{oo1}\\
\lim_{n\to\infty}\mD(u_n)=\mD^*.\label{oo2}
\end{gather}
Up to a subsequence we may also assume that
\begin{align*}
(\mM(u_n),\mH(u_n))\to(\mM_0,\mH_0)\quad\text{as $n\to\infty$}.
\end{align*}
By continuity of $\mD$ and finiteness of $\mD^*$ we know that
\begin{align*}
\mD^*=\mD(\mM_0,\mH_0),\quad
\mM_0\in[0,c),\quad
\mH_0\in[0,m_{c}-\nu).
\end{align*}
From Lemma \ref{lemma coercivity} and \ref{cnls killip visan curve} it follows that $(u_n(0))_n$ is a bounded sequence in $H_{x,y}^1$, hence Lemma \ref{linear profile} is applicable for $(u_n(0))_n$: There exist nonzero linear profiles
$(\tdu^j)_j\subset H_{x,y}^1$, remainders $(w_n^k)_{k,n}\subset H_{x,y}^1$, parameters $(t^j_n,x^j_n)_{j,n}\subset\R\times\R^2$ and $K^*\in\N\cup\{\infty\}$, such that
\begin{itemize}
\item[(i)] For any finite $1\leq j\leq K^*$ the parameter $t^j_n$ satisfies
\begin{align}
t^j_n\equiv 0\quad\text{or}\quad \lim_{n\to\infty}t^j_n= \pm\infty.
\end{align}

\item[(ii)]For any finite $1\leq k\leq K^*$ we have the decomposition
\begin{align}\label{cnls decomp}
u_n(0)=\sum_{j=1}^k T_n^j\phi^j(x,y)+w_n^k=:\sum_{j=1}^k e^{-it_n\Delta_x}\phi^j(x-x_n,y)+w_n^k.
\end{align}

\item[(iii)] The remainders $(w_n^k)_{k,n}$ satisfy
\begin{align}\label{cnls to zero wnk}
\lim_{k\to K^*}\lim_{n\to\infty}\|e^{it\Delta_{x,y}}w_n^k\|_{L_{t,x}^4 H_y^s(\R)}=0.
\end{align}

\item[(iv)] The parameters are orthogonal in the sense that
\begin{align}\label{cnls orthog of pairs}
|t_n^k-t_n^j|+|x_n^k-x_n^j|\to\infty
\end{align}
for any $j\neq k$.

\item[(v)] For any finite $1\leq k\leq K^*$ and $D\in\{1,\pt_{x_i},\pt_y\}$ we have the energy decompositions
\begin{align}
\|D(u_n(0))\|_{2}^2&=\sum_{j=1}^k\|D(T_n^j\tdu^j)\|_{2}^2+\|Dw_n^k\|_{2}^2+o_n(1),\label{orthog L2}\\
\|u_n(0)\|_{4}^{4}&=\sum_{j=1}^k\|T_n^j\tdu^j\|_{4}^{4}
+\|w_n^k\|_{4}^{4}+o_n(1)\label{cnls conv of h}.
\end{align}
\end{itemize}
We now define the nonlinear profiles as follows:
\begin{itemize}
\item For $t^k_\infty=0$, we define $u^k$ as the solution of \eqref{nls} with $u^k(0)=\tdu^k$.

\item For $t^k_\infty\to\pm\infty$, we define $u^k$ as the solution of \eqref{nls} that scatters forward (backward) to $e^{-it\Delta_{x,y}}\tdu^k$ in $H_{x,y}^1$.
\end{itemize}
In both cases we define
\begin{align*}
u_n^k:=u^k(t-t^k_n,x-x_n^k,y).
\end{align*}
Then $u_n^k$ is also a solution of \eqref{nls}. In both cases we have for each finite $1\leq k \leq K^*$
\begin{align}\label{conv of nonlinear profiles in h1}
\lim_{n\to\infty}\|u_n^k(0)-T_n^k \tdu^k\|_{H_{x,y}^1}=0.
\end{align}

In the following, we establish a Palais-Smale type lemma which is essential for the construction of the minimal blow-up solution.

\begin{lemma}[Palais-Smale-condition]\label{Palais Smale}
Let $(u_n)_n$ be a sequence of solutions of \eqref{nls} with maximal lifespan $I_n$, $u_n\in\mA$ and $\lim_{n\to\infty}\mD(u_n)=\mD^*$. Assume also that there exists a sequence $(t_n)_n\subset\prod_n I_n$ such that
\begin{align}\label{precondition}
\lim_{n\to\infty}\|u_n\|_{L_{t,x}^4 H_y^{1}((\inf I_n,\,t_n])}=\lim_{n\to\infty}\|u_n\|_{L_{t,x}^4 H_y^{1}([t_n,\,\sup I_n)}=\infty.
\end{align}
Then up to a subsequence, there exists a sequence $(x_n)_n\subset\R^2$ such that $(u_n(t_n, \cdot+x_n,y))_n$ strongly converges in $H_{x,y}^1$.
\end{lemma}

\begin{proof}
The proof essentially follows the same steps from the proof of \cite[Lem. 3.16]{Luo_Waveguide_MassCritical}. Here we will show that the Step 1 in the proof of \cite[Lem. 3.16]{Luo_Waveguide_MassCritical} will continue to hold in the context of the so far constructed variational framework, and we refer the remaining proof to the Steps 2 to 4 from the proof of \cite[Lem. 3.16]{Luo_Waveguide_MassCritical} (where Lemma \ref{lem stability cnls} and \ref{cnls lem large scale proxy} are applied in these steps).

More precisely, we need to show that for any nonzero linear profiles $\phi^j$ and nonzero remainders $w_n^j$ we have
\begin{alignat}{2}
\mH(T_n^j\phi^j)&> 0,\quad E(w_n^j)&&>0\label{bd for S}\\
\mK(T_n^j\phi^j)&> 0,\quad K(w_n^j)&&>0\label{pos of K}
\end{alignat}
for all sufficiently large $n=n(j)\in\N$. We only show the statement for the atoms $\phi^j$, the one for the remainder $w_n^j$ follows in the same manner.

Let us first prove that $m_c$ can also be characterized as
\begin{align}
m_c=\inf\{I(u):u\in S(c),\,K(u)\leq 0\}=:\bar m_c,\label{I charac bar}
\end{align}
where $I(u)$ is defined by \eqref{1.4+}. By definition it is clear that $m_c\geq \bar m_c$. On the other hand, if $u\in V(c)$ then
\[I(u)=E(u)-\frac12 K(u)=E(u)\geq m_c.\]
Otherwise, let $u\in S(c)$ satisfy $K(u)<0$. Then there exists some $t\in(0,1)$ such that $K(tu)=0$. Using Proposition \ref{prop monotone} we then obtain
\[I(u)\geq I(tu)=I(tu)+\frac12 K(tu)=E(tu)\geq m_{t^2 M(u)}\geq m_c,\]
from which the claim follows.

We next assume that \eqref{pos of K} does not hold. Then up to a subsequence we may assume that there exists some $\phi^j$ such that $K(T_n^j\phi^j)\leq 0$ for all $n\in\N$. By the mass decomposition \eqref{orthog L2} we know that $M(T_n^j)<c$ for all $n\gg 1$. Notice that using \eqref{orthog L2} and \eqref{cnls conv of h} we also have the energy decomposition of $I(u_n)$:
\begin{align}
\mI(u_n(0))&=\sum_{j=1}^k\mI(T_n^j\tdu^j)+\mI(w_n^k)+o_n(1)\label{conv of i}.
\end{align}
By the non-negativity of the functional $I$, Proposition \ref{prop monotone}, \eqref{I charac bar} and the inductive hypothesis (implying $\limsup_{n\to\infty} I(u_n)<m_c$) we obtain
\begin{align*}
m_c\leq m_{M(T_n^j\phi^j)}\leq I(T_n^j\phi^j)\leq m_c-\delta
\end{align*}
for some small positive $\delta\ll 1$ and all $n\gg 1$, which leads to a contradiction. This completes the proof.
\end{proof}

Having proved Lemma \ref{Palais Smale} it will be a standard routine to prove the following lemmas concerning the existence of a minimal blow-up solution $u_c$ and its compactness properties. We omit their standard proofs and refer e.g. to \cite[Lem. 4.19, Lem. 4.20]{Luo_Waveguide_MassCritical} for complete details.

\begin{lemma}[Existence of a minimal blow-up solution]\label{category 0 and 1}
Suppose that $\mD^*\in(0,\infty)$. Then there exists a global solution $u_c$ of \eqref{nls} such that $\mD(u_c)=\mD^*$ and
\begin{align*}
\|u_c\|_{L_{t,x}^4 H_y^1((-\infty,0])}=\|u_c\|_{L_{t,x}^4 H_y^1([0,\infty))}=\infty.
\end{align*}
Moreover, $u_c$ is almost periodic in $H_{x,y}^1$ modulo $\R_x^2$-translations, i.e. the set $\{u(t):t\in\R\}$ is precompact in $H_{x,y}^1$ modulo translations w.r.t. the $x$-variable.
\end{lemma}

\begin{lemma}\label{lemma property of uc}
Let $u_c$ be the minimal blow-up solution given by Lemma \ref{category 0 and 1}. Then
\begin{itemize}
\item[(i)] There exists a center function $x:\R\to \R^2$ such that for each $\vare>0$ there exists $R>0$ such that
\begin{align*}
\int_{|x+x(t)|\geq R}|\nabla_{x,y} u_c(t)|^2+|u_c(t)|^2+|u_c(t)|^{4}\,dxdy\leq\vare\quad\forall\,t\in\R.
\end{align*}

\item[(ii)]There exists some $\delta>0$ such that $\inf_{t\in\R}\mK(u_c(t))=\delta$.
\end{itemize}
\end{lemma}

\subsection{Extinction of the minimal blow-up solution}
In this subsection we close the proof of Theorem \ref{thm large data} by showing the contradiction that the minimal blow-up solution $u_c$ must be equal to zero. The proof is based on the compactness properties of the critical element $u_c$ deduced in last section, and the following lemma on the decay property of the center function. For a proof, we refer e.g. to \cite{non_radial} or \cite{killip_visan_soliton}.

\begin{lemma}\label{holmer}
Let $x(t)$ be the center function given by Lemma \ref{lemma property of uc}. Then $x(t)$ obeys the decay condition $x(t)=o(t)$ as $|t|\to\infty$.
\end{lemma}

We are now ready to prove Theorem \ref{thm large data}.

\begin{proof}[Proof of Theorem \ref{thm large data}]
We firstly prove the statement concerning the global well-posedness of \eqref{nls} below ground states. Since \eqref{nls} is energy-subcritical, it is well-known (see for instance \cite{TzvetkovVisciglia2016}) that the global well-posedness of a solution $u$ of \eqref{nls}is equivalent to the statement that for all $n\in\N$ we have
\begin{align}\label{7.44}
\sup_{t\in[-n,n]}\|\nabla_{x,y}u(t)\|_{2}<\infty.
\end{align}
Let $u_0$ satisfy \eqref{gwp condition} and set $c:=M(u_0)$. By Lemma \ref{lemma coercivity} and the conservation of mass and energy we obtain
\[\|u\|_{H_{x,y}^1}^2\lesssim M(u)+E(u)<c+m_c-\nu,\]
from which \eqref{7.44} and consequently the global well-posedness of $u$ follows.

Let us now restrict ourselves to the case $d=2$ and $M(u_0)<\pi \widehat M(Q)$ in order to prove the large data scattering result below ground states. As mentioned previously, we show the contradiction that the minimal blow-up solution $u_c$ given by Lemma \ref{category 0 and 1} is equal to zero. Let $\chi:\R^d\to\R$ be a smooth radial cut-off function satisfying
\begin{align*}
\chi=\left\{
             \begin{array}{ll}
             |x|^2,&\text{if $|x|\leq 1$},\\
             0,&\text{if $|x|\geq 2$}.
             \end{array}
\right.
\end{align*}
For $R>0$, we define the local virial action $z_R(t)$ by
\begin{align*}
z_{R}(t):=\int R^2\chi\bg(\frac{x}{R}\bg)|u_c(t,x)|^2\,dxdy.
\end{align*}
Direct computation yields
\begin{align}
\pt_t z_R(t)=&\,2\,\mathrm{Im}\int R\nabla_x \chi\bg(\frac{x}{R}\bg)\cdot\nabla_x u_c(t)\bar{u}_c(t)\,dxdy,\label{final4}
\end{align}
and
\begin{align}
\pt_t^2 z_R(t)=&\,4\int \pt^2_{x_j x_k}\chi\bg(\frac{x}{R}\bg)\pt_{x_j} u_c\pt_{x_k}\bar{u}_c\,dxdy-\frac{1}{R^2}\int\Delta_x^2\chi\bg(\frac{x}{R}\bg)|u_c|^2\,dxdy\nonumber\\
-&\,\int\Delta_x\chi\bg(\frac{x}{R}\bg)|u_c|^{\alpha+2}\,dxdy\nonumber.
\end{align}
We then obtain
\begin{align*}
\pt_t^2 z_R(t)=8\mK(u_c(t))+A_R(u_c(t)),
\end{align*}
where
\begin{align*}
A_R(u_c(t))=&\,4\int\bg(\pt^2_{x_j}\chi\bg(\frac{x}{R}\bg)-2\bg)|\pt_{x_j} u_c|^2\,dxdy+4\sum_{j\neq k}\int_{R\leq|x|\leq 2R}\pt_{x_j}\pt_{x_k}\chi\bg(\frac{x}{R}\bg)\pt_{x_j} u_c\pt_{x_k}\bar{u}_c\,dxdy\nonumber\\
-&\,\frac{1}{R^2}\int\Delta_x^2\chi\bg(\frac{x}{R}\bg)|u_c|^2\,dxdy
-\int\bg(\Delta_x\chi\bg(\frac{x}{R}\bg)-2d\bg)|u_c|^{\alpha+2}\,dxdy.
\end{align*}
We have the rough estimate
\begin{align*}
|A_R(u(t))|\leq C_1\int_{|x|\geq R}|\nabla_x u_c(t)|^2+\frac{1}{R^2}|u_c(t)|^2+|u_c(t)|^{\alpha+2}\,dxdy
\end{align*}
for some $C_1>0$. By Lemma \ref{lemma property of uc} we know that there exists some $\delta>0$ such that
\begin{align}\label{small extinction ff}
\inf_{t\in\R}(8\mK(u_c(t)))\geq 8\delta=:2\eta_1>0.
\end{align}
From Lemma \ref{lemma property of uc} it also follows that there exists some $R_0\geq 1$ such that
\begin{align*}
\int_{|x+x(t)|\geq R_0}|\nabla_{x,y} u_c(t)|^2+|u_c(t)|^2+|u_c(t)|^{\alpha+2}\,dxdy\leq \frac{\eta_1}{C_1}.
\end{align*}
Thus for any $R\geq R_0+\sup_{t\in[t_0,t_1]}|x(t)|$ with some to be determined $t_0,t_1\in[0,\infty)$, we have
\begin{align}\label{final3}
\pt_t^2 z_R(t)\geq \eta_1
\end{align}
for all $t\in[t_0,t_1]$. By Lemma \ref{holmer} we know that for any $\eta_2>0$ there exists some $t_0\gg 1$ such that $|x(t)|\leq\eta_2 t$ for all $t\geq t_0$. Now set $R=R_0+\eta_2 t_1$. Integrating \eqref{final3} over $[t_0,t_1]$ yields
\begin{align}\label{12}
\pt_t z_R(t_1)-\pt_t z_R(t_0)\geq \eta_1 (t_1-t_0).
\end{align}
Using \eqref{final4}, Cauchy-Schwarz and Lemma \ref{cnls killip visan curve} we have
\begin{align}\label{13}
|\pt_t z_{R}(t)|\leq C_2 \mD^*R= C_2 \mD^*(R_0+\eta_2 t_1)
\end{align}
for some $C_2=C_2(\mD^*)>0$. \eqref{12} and \eqref{13} give us
\begin{align*}
2C_2 \mD^*(R_0+\eta_2 t_1)\geq\eta_1 (t_1-t_0).
\end{align*}
Setting $\eta_2=\frac{\eta_1}{4C_2\mD^*}$, dividing both sides by $t_1$ and then sending $t_1$ to infinity we obtain $\frac{1}{2}\eta_1\geq\eta_1$, which implies $\eta_1\leq 0$, a contradiction. This completes the proof.
\end{proof}

\subsection{Finite time blow-up below ground states}
\begin{proof}[Proof of Theorem \ref{thm blow up}]
We follow Glassey's virial arguments \cite{Glassey1977} to prove the claim. First notice that by using the same approximation arguments as in the proof of \cite[Prop. 6.5.1]{Cazenave2003} we are able to show that $|x|u(t)\in L_{x,y}^2$ for all $t$ in the maximal lifespan of $u$. This enable us to define the quantity
\begin{align*}
J(t):=\int |x|^2|u(t,x,y)|^2\,dxdy
\end{align*}
and direct computation yields $\pt_{t}^2J(t)=8\mK(u(t))$.

Next, define the set $\mathcal{B}$ by
\[\mathcal{B}:=\{\phi\in H_{x,y}^1: M(\phi)<2\pi\widehat M(Q),\,E(u)<m_{M(\phi)},\,K(\phi)<0\}.\]
We claim that for any $\phi\in\mathcal{B}$ we have
\begin{align}\label{blow up upper}
K(\phi)<2(E(u)-m_{M(\phi)}).
\end{align}
Indeed, since $K(\phi)<0$, there exists some $t\in(0,1)$ such that $K(t\phi)=0$. Using also the monotonicity of $c\mapsto m_c$ we obtain
\[m_{M(\phi)}\leq m_{t^2M(\phi)}=\frac{t^2}{2}\|\pt_y\phi\|_2^2<\frac12\|\pt_y\phi\|_2^2=E(u)-\frac12 K(u),\]
from which the claim follows.

We now come back to the main proof. Arguing as in the proof of Lemma \ref{lem 7.8} we know that if $u_0\in \mathcal{B}$, so is $u(t)$ for all $t$ in the lifespan of $\mathcal{B}$. Thus \eqref{blow up upper} and $\pt_{t}^2 J(t)=8\mK(u(t))$ yield
\begin{align*}
\pt_{t}^2 J(t)=8\mK(u(t))\leq 16(\mH(u)-m_{\mM(u)})<0.
\end{align*}
This particularly implies that $t\mapsto J(t)$ is a positive and concave function simultaneously. Hence the function $t\mapsto J(t)$ can not exist for all $t\in\R$ and the desired claim follows.
\end{proof}




\end{document}